\documentclass[11pt]{article}

\usepackage[utf8]{inputenc}
\usepackage[english]{babel}
\usepackage{amsmath}
\usepackage{nicefrac}
\usepackage{amsthm}
\usepackage{amsfonts}
\usepackage{verbatim}
\usepackage{color}
\usepackage[arrow, matrix, curve]{xy}
\usepackage{bbm}
\usepackage{eqparbox}
\usepackage[numbers]{natbib}
\usepackage{stmaryrd}
\usepackage{amssymb}
\usepackage{mathrsfs}
\usepackage{pictexwd,dcpic}
\usepackage{mathabx}

\DeclareMathSymbol{\leq}{\mathrel}{symbols}{20}
   
\DeclareMathSymbol{\geq}{\mathrel}{symbols}{21}

\newtheoremstyle{WreschTheoremstyle} 
                        {1.5em}    
                        {2.5em}    
                        {}         
                        {}         
                        {\bfseries}
                        {}        
                        {\newline} 
                        {\raisebox{0.6em}{\thmname{#1}\thmnumber{#2}\thmnote{ (#3)}}}

\newcommand{\R}{\mathbb{R}}

\newcommand{\N}{\mathbb{N}}

\newcommand{\F}{\mathcal{F}}

\renewcommand{\P}{\mathbb{P}}
\newcommand{\E}{\mathbb{E}}
\newcommand{\e}{\varepsilon}

\renewcommand{\1}{\mathbbm{1}}

\newtheorem{Theorem}{Theorem}[]
\newtheorem{Proposition}[Theorem]{Proposition}
\newtheorem{Corollary}[Theorem]{Corollary}
\newtheorem{Lemma}[Theorem]{Lemma}
\newtheorem{Remark}[Theorem]{Remark}

\newtheorem{Example}[Theorem]{Example}

\numberwithin{equation}{section}

\makeatletter
\newcommand{\customlabel}[1]{%
     \stepcounter{ref}%
   \protected@write
\@auxout{}{\string\newlabel{#1}{{\thesatz.\arabic{ref}}{\thepage}{\thesatz.\arabic{ref}}{#1}{}}}%
   \hypertarget{#1}{\thesatz.\arabic{ref}}%
}
\makeatother

\newenvironment{sciabstract}{\begin{quote}}{\end{quote}}

\newcounter{lastnote}

\title{Existence of densities for multi-type CBI processes}
\newcommand{\pdftitle}{Existence of densities for multi-type CBI processes}
\newcommand{\pdfauthor}{Martin Friesen}

\author{
 Martin Friesen\footnote{Fakult\"at f\"ur Mathematik und Naturwissenschaften, Bergische Universit\"at Wuppertal, Gaußstraße 20, 42119 Wuppertal, Germany, friesen@math.uni-wuppertal.de}\\
 Peng Jin\footnote{Fakult\"at f\"ur Mathematik und Naturwissenschaften, Bergische Universit\"at Wuppertal, Gaußstraße 20, 42119 Wuppertal, Germany, jin@uni-wuppertal.de}\\
 Barbara R\"udiger\footnote{Fakult\"at f\"ur Mathematik und Naturwissenschaften, Bergische Universit\"at Wuppertal, Gaußstraße 20, 42119 Wuppertal, Germany, ruediger@uni-wuppertal.de}
}

\usepackage[plainpages=false,pdfpagelabels=true,bookmarks=true,pdfauthor={\pdfauthor},
pdftitle={\pdftitle}]{hyperref}

\makeatletter
\def\HyPsd@CatcodeWarning#1{}
\makeatother

\begin{document}




\maketitle

\begin{sciabstract}\textbf{Abstract:}
Let $X$ be a multi-type continuous-state branching process with immigration (CBI process) on state space $\R_+^d$.
Denote by $g_t$, $t \geq 0$, the law of $X(t)$.
We provide sufficient conditions under which $g_t$ has, for each $t > 0$, a density with respect to the Lebesgue measure.
Such density has, by construction, some anisotropic Besov regularity.
Our approach neither relies on the use of Malliavin calculus nor on the study of corresponding Laplace transform.
\end{sciabstract}


\section{Introduction}
Multi-type CBI processes are Markov processes with state space
\[
 \R_+^d = \{ x \in \R^d \ | \ x_1,\dots, x_d \geq 0\}, \ \ d \in \N,
\]
which arise as scaling limits of Galton-Watson branching processes with immigration, see, e.g., \cite{L06, L11}.
A remarkable feature of multi-type CBI processes is that the logarithm of their Laplace transform is an affine function of the initial state variable,
i.e., multi-type CBI processes are affine processes in the sense of \cite[Definition 2.6]{DFS03}.
They are also semimartingales whose characteristics
can be readily deduced from their branching and immigration mechanisms.
Although these processes are primarily motivated by population models,
they have also found many applications in finance, especially in term-structure interest rate
models and stochastic volatility models, see, e.g., \cite{DFS03}.

Let us describe these processes in more detail.
According to \cite[Theorem 2.7]{DFS03} (see also \cite[Remark 2.5]{BLP15}),
there exists a unique conservative Feller semigroup $(P_{t})_{t \geq 0}$
acting on the Banach space of continuous functions vanishing at infinity
with state space $\R_{+}^{d}$ such that its infinitesimal generator
has core $C_{c}^{2}(\R_{+}^{d})$ and is, for $f\in C_{c}^{2}(\R_{+}^{d})$,
given by
\begin{align}\label{GENERATOR}
 (Lf)(x) &= \sum \limits_{i=1}^{d}c_i x_i \frac{\partial^2 f(x)}{\partial x_i^2} + (\beta + Bx)\cdot (\nabla f)(x)
 + \int \limits_{\R_+^d}( f(x+z) - f(x) ) \nu(dz)
 \\ \notag &\ \ \ + \sum \limits_{i=1}^{d}x_i \int \limits_{\R_+^d} \left( f(x+z) - f(x) - \frac{\partial f(x)}{\partial x_i}(1 \wedge z_i) \right)\mu_i(dz),
\end{align}
provided that the tuple $(c,\beta,B,\nu,\mu)$ satisfies
\begin{enumerate}
 \item[(i)] $c = (c_1,\dots, c_d) \in \R_+^d$.
 \item[(ii)] $\beta = (\beta_1,\dots, \beta_d) \in \R_+^d$.
 \item[(iii)] $B = (b_{ij})_{i,j \in \{1,\dots, d\}}$ is such that $b_{ij} \geq 0$ whenever $i,j \in \{1,\dots, d\}$ satisfy $i \neq j$.
 \item[(iv)] $\nu$ is a Borel measure on $\R_+^d$ satisfying $\int_{\R_+^d}( 1 \wedge |z|) \nu(dz) < \infty$ and $\nu(\{0\}) = 0$.
 \item[(vi)] $\mu = (\mu_1,\dots, \mu_d)$, where, for each $i \in \{1,\dots, d\}$, $\mu_i$ is a Borel measure on $\R_+^d$ satisfying
\begin{align}\label{EQ:00}
 \int \limits_{\R_+^d} \left( |z| \wedge |z|^2 + \sum \limits_{j \in \{1,\dots, d\} \backslash \{i\}} (1 \wedge z_j) \right) \mu_i(dz) < \infty,
\ \ \mu_i(\{0\}) = 0.
\end{align}
\end{enumerate}
The corresponding Markov process with generator $L$ is called a (conservative) multi-type CBI process. We call a
tuple $(c,\beta,B,\nu,\mu)$ with properties (i) – (vi) admissible.
Note that this notion of admissible parameters is
a special case of \cite[Definition 2.6]{DFS03}, see also \citep[Remark 2.3]{BLP15}
for additional comments. One of the advantages of multi-type CBI processes
is their analytical tractability via Laplace transforms. More precisely,
the Laplace transform of the transition semigroup $P_{t}$ can be
computed explicitly in terms of solutions to generalized Riccati equations.
Most of the results obtained for multi-type CBI processes are based
on a detailed study of these equations.

In this work we study existence of (transition) densities for multi-type
CBI processes. A general expository on one-dimensional CBI processes was recently given in \cite{CLP18}, while a particular
example of a two-dimensional affine process was studied
in \cite{JKR17}. Both approaches essentially rely on the study of
the corresponding Riccati equations, i.e. on the Laplace transform
of the transition semigroup. Results applicable for a wide class of
affine processes on the state space $\R^{n}\times\R_{+}^{d}$ were
obtained in \cite{FMS13}. Applying their main result to the particular
case of multi-type CBI processes requires that $c_{1},\dots,c_{d}>0$,
i.e. the diffusion component is non-degenerate. Results applicable
also to cases without diffusion (i.e. $c_{1}=\dots=c_{d}=0$) are,
to the best of our knowledge, not available in arbitrary dimension.
Such results should, of course, rely on the smoothing property of
jumps corresponding to the branching and immigration mechanisms. We
would like to mention that, similar to the diffusion case, there
also exists a Malliavin calculus for stochastic equations with jumps
\cite{BC86,P96,P97}. It is, however, much less powerful then its
counterpart for diffusions.

We use some ideas developed in \cite{FP10, DF13, R17}, which provide a simple
technique to prove existence of a density having some Besov-regularity without the use of Malliavin calculus.
Their techniques were applied to L\'evy driven stochastic equations with H\"older continuous coefficients \cite{DF13},
3D Navier-Stokes equations driven by Gaussian noise \cite{DR14}, but also to the space-homogeneous Boltzmann equation \cite{F15}.

\section{Statement of the results}

\subsection{The anisotropic Besov space}
Due to (\ref{GENERATOR}) and the abundant choice
of admissible parameters, it is reasonable to expect that the different
components of a multi-type CBI processes on $\R_{+}^{d}$ behave very
differently. Below we introduce anisotropic Besov spaces, which enable
us to measure regularity for the density of each component of a CBI
process separately. Similar ideas have been also applied in \cite{FJR18} to stochastic equations driven by L\'evy processes with anisotropic jumps.
We call $a = (a_1,\dots, a_d)$ an anisotropy if it satisfies
\begin{align}\label{EQ:35}
 0 < a_1, \dots, a_d < \infty \qquad \text{ and } \qquad a_1 + \dots + a_d = d.
\end{align}
For $\lambda > 0$ with $\lambda / a_k \in (0,1)$, $k = 1,\dots, d$, the anisotropic Besov space $B_{1,\infty}^{\lambda, a}(\R^d)$ is
defined as the Banach space of functions $f: \R^d \longrightarrow \R$ with finite norm
\begin{align}\label{EQ:36}
 \| f\|_{B_{1,\infty}^{\lambda,a}} := \| f\|_{L^1(\R^d)} +
 \sum \limits_{k=1}^{d}\sup \limits_{h\in  [-1,1]} |h|^{- \lambda / a_k} \| \Delta_{h e_k}f \|_{L^1(\R^d)},
\end{align}
where $\Delta_h f(x) = f(x+h) - f(x)$, $h \in \R^d$, see \cite{D03} and \cite{T06} for additional details and references.
Here $e_1,\dots, e_d$ denote the canonical basis vectors in $\R^d$.
In the above definition, $\lambda / a_k$ describes the smoothness in the coordinate $k$,
its restriction to $(0,1)$ is not essential.
Without this restriction we should use iterated differences in \eqref{EQ:36} instead (see \cite[Theorem 5.8.(ii)]{T06}).

\subsection{Smoothing property of the noise}
Let $(c,\beta,B,\nu,\mu)$ be admissible parameters.
For given $x\in\R^{d}$, let $L^{x}=(L^{x}(t))_{t \geq 0}$ be a Lévy
process on $\R^{d}$ whose characteristic function $\E[e^{i\lambda\cdot L^{x}(t)}]=e^{-t\Psi_{x}(\lambda)}$,
$\lambda\in\R^{d}$, satisfies
\begin{align}
\Psi_{x}(\lambda) & =\sum\limits _{j=1}^{d}2c_{j}x_{j}\1_{\R_{+}}(x_{j})\lambda_{j}^{2}+\int\limits _{\R_{+}^{d}}\left(1-e^{i\lambda\cdot z}\right)\nu(dz)\label{EQ:01}\\
 & \ \ \ +\sum\limits _{j=1}^{d}\1_{\R_{+}}(x_{j})x_{j}\int\limits _{|z|\leq1}\left(1+i\lambda\cdot z-e^{i\lambda\cdot z}\right)\mu_{j}(dz).\nonumber
\end{align}
Denote by $g_{t}^{x}(dz)$ the distribution of $L^{x}(t)$.
If this distribution has a density with respect to the Lebesgue measure,
then, by abuse of notation, we denote this density also by $g_{t}^{x}(z)$.
Let $(\alpha_{i})_{i\in\{1,\dots,d\}}\subset(0,2]$.
For $I\subset\{1,\dots,d\}$, define
\begin{align}
\rho_{I}(x):=\min\{x_{j}^{1/\alpha_{j}}\ |\ j\in I\}\1_{\R_{+}^{d}}(x),\qquad\rho_{\emptyset}=\1_{\R_{+}^{d}}(x),\qquad\Gamma(I)=\{x\in\R^{d}\ |\ \rho_{I}(x)>0\}.\label{EQ:05}
\end{align}
The following is our main condition on the smoothing property of the
noise.
\begin{enumerate}
\item[(A)] There exists $I\subset\{1,\dots,d\}$ and constants $(\alpha_{i})_{i\in\{1,\dots,d\}}\subset(0,2]$,
$C,t_{0}>0$ such that, for each $x\in\Gamma(I)$ and $t\in(0,t_{0})$,
the distribution $g_{t}^{x}$ has a density with respect to the Lebesgue
measure satisfying, for any $i\in\{1,\dots,d\}$ and $t\in(0,t_{0})$,
\begin{align}
\int\limits _{\R^{d}}|g_{t}^{x}(z+he_{i})-g_{t}^{x}(z)|dz\leq\frac{C|h|}{\rho_{I}(x)}t^{-1/\alpha_{i}},\ \ h\in[-1,1].\label{MAIN:ASSUMPTION}
\end{align}
\end{enumerate}
Here $\alpha_{i}$ describes the smoothness of the noise. These constants are related
with an anisotropy $a=(a_{i})_{i\in\{1,\dots,d\}}$ and a mean order
of smoothness $\overline{\alpha}$ by
\begin{align}
\frac{1}{\overline{\alpha}}=\frac{1}{d}\left(\frac{1}{\alpha_{1}}+\cdots+\frac{1}{\alpha_{d}}\right),\qquad a_{i}=\frac{\overline{\alpha}}{\alpha_{i}},\ \ i\in\{1,\dots,d\}.\label{APPROX:03}
\end{align}
Hence larger values for $\alpha_{i}$ give higher smoothness, that is, larger values for $\overline{\alpha}$. The factor
$\rho_{I}$ is essential to treat the boundary behavior of multi-type
CBI processes. By convention $1/0:=+\infty$, we see that \eqref{MAIN:ASSUMPTION}
is clearly satisfied, if $\rho_{I}(x)=0$, i.e. $x \not \in \Gamma(I)$.
In Section 6 we provide some sufficient conditions for (A).
Based on these conditions, below we provide our main guiding examples.
\begin{Example}\label{EXAMPLE:00}
 \item[(a)] Define $I_{1}=\left\{ j\in\{1,\dots,d\}\ |\ c_{j}>0\right\} $ and
 let $I_{2}:=\{1,\dots,d\}\backslash I_{1}$. Suppose that, for each
 $j\in I_{2}$, the L\'evy measure $\mu_j$ satisfies
 \[
  \mu_j(dz) = \1_{\R_+}(z_j)\frac{dz_j}{z_j^{1+ \alpha_j}} \otimes \prod \limits_{k \neq j} \delta_0(dz_k), \ \ \alpha_j \in (1,2).
 \]
 Then (A) is satisfied for $I=\{1,\dots,d\}$ and $\alpha_{j}=2\1_{I_{1}}(j)+\alpha_{j}\1_{I_{2}}(j)$,
 see Lemma \ref{LEMMA:SMOOTHING}.
 \item[(b)] It is worthwhile to mention that the particular choice
\[
 \mu_j(dz) = \1_{\R_+^d}(z)\frac{dz}{ |z|^{d + \alpha} }, \ \ \alpha \in (0,2)
\]
 violates \eqref{EQ:00} for any choice of $\alpha \in (0,2)$.
 However, suppose that there exists $j \in \{1,\dots, d\}$ and a L\'evy measure $\mu_j'$ on $\R_+^d$ such that
 \[
  \mu_j(dz) = \1_{\R_+^d}(z) \1_{ \{ |z| \leq 1\}}  \frac{dz}{|z|^{d+\alpha}} + \mu_j'(dz), \ \ \alpha \in (0,1),
 \]
 then condition (A) is satisfied for $I = \{j \}$ and $\alpha_1 = \dots = \alpha_d = \alpha$, see Proposition \ref{PROP:03} and
 Lemma \ref{LEMMA:SMOOTHING}.
 Nevertheless this example does not satisfy the other restrictions formulated in our main results below.
\item[(c)] Suppose that there exists a subordinator $\nu'$ on $\R_+^d$ such that
 \[
  \nu(dz) = \1_{\R_+^d}(z) \1_{ \{ |z| \leq 1\}} \frac{dz}{|z|^{d+\alpha}} + \nu'(dz), \ \ \alpha \in (0,1),
 \]
 then condition (A) is satisfied for $I = \emptyset$ and $\alpha_1 = \dots = \alpha_d = \alpha$, see Lemma \ref{LEMMA:SMOOTHING1}.
\end{Example}
It is worthwhile to mention that we may also consider more general classes branching and immigration measures which include, in particular, cases where $\mu_j$ and $\nu$ are not absolutely continuous with respect to the Lebesgue measure, see \cite{FJR18} for additional details.

\subsection{Existence of densities for multi-type CBI processes}

We start with the most general case and then continue with more specific situations. 
\begin{Theorem}\label{MAINTHEOREM} 
Let $X$ be a multi-type CBI process with admissible parameters $(c,\beta,B,\nu,\mu)$
and suppose that
\begin{enumerate}
\item[(a)] Condition (A) holds for $I=\{1,\dots,d\}$ and some $\alpha_{1},\dots,\alpha_{d}>\frac{4}{3}$.
\item[(b)] There exists $\tau\in(0,1)$ such that
\[
\sum\limits _{j=1}^{d}\int\limits _{|z|>1}|z|^{1+\tau}\mu_{j}(dz)+\int\limits _{|z|>1}|z|^{1+\tau}\nu(dz)<\infty.
\]
\end{enumerate}
If $X(0)$ satisfies $\E[|X(0)|^{1 + \tau}] < \infty$, then for each $t>0$, $X(t)$ has a density $g_{t}$ on
\[
\Gamma(\{1,\dots,d\})=\{x\in\R_{+}^{d}\ |\ x_{1},\dots,x_{d}>0\}.
\]
Moreover, $f_{t}(x):=\rho(x)g_{t}(x)\in B_{1,\infty}^{\lambda,a}(\R^{d})$,
where $\lambda>0$ is small enough, $a$ is defined by \eqref{APPROX:03}
and $\rho(x)=\min\{x_{1}^{1/\alpha_{1}},\dots,x_{d}^{1/\alpha_{d}}\}\1_{\R_{+}^{d}}(x)$.
\end{Theorem} 
This statement is, e.g., applicable in the situation of Example \ref{EXAMPLE:00}.(a),
where the smoothing property (A) is obtained from a combination of diffusion and jumps from the branching mechanism. 
In absence of diffusion, we can weaken the restriction on $\alpha_{1},\dots,\alpha_{d}$ slightly. 
\begin{Theorem}\label{MAINTHEOREM1}
Let $X$ be a multi-type CBI process with admissible parameters $(c,\beta,B,\nu,\mu)$, where $c_{1}=\dots=c_{d}=0$, and suppose that
\begin{enumerate}
\item[(a)] Condition (A) is satisfied for some $I\subset\{1,\dots,d\}$ and
$\alpha_{1},\dots,\alpha_{d}\in(0,2)$.
\item[(b)] There exists $\gamma_{0}\in(1,2]$ and $\tau\in(0,\gamma_{0}-1)$
such that
\[
\sum\limits _{j=1}^{d}\int\limits _{\R_{+}^{d}}\left(|z|^{\gamma_{0}}\1_{\{|z|\leq1\}}+|z|^{1+\tau}\1_{\{|z|>1\}}\right)\mu_{j}(dz)+\int\limits _{|z|>1}|z|^{1+\tau}\nu(dz)<\infty.
\]
If $I=\emptyset$, then we may also take $\tau=0$.
\item[(c)] It holds that $\alpha_{1},\dots,\alpha_{d}>\frac{\gamma_{0}}{1+\gamma_{0}}\gamma_{0}$.
Moreover, for each $j\in I$, one has $\alpha_{j}\geq1$.
\end{enumerate}
If $X(0)$ satisfies $\E[|X(0)|^{1 + \tau}] < \infty$, then for each $t>0$, 
$X(t)$ has a density $g_{t}$ on $\Gamma(I)$. 
Moreover, $f_{t}(x):=\rho_{I}(x)g_{t}(x)\in B_{1,\infty}^{\lambda,a}(\R^{d})$,
where $\lambda>0$ is small enough, $a$ is defined in \eqref{APPROX:03} and $\rho_{I}$ is given as in \eqref{EQ:05}. 
\end{Theorem}
We now make a few comments on Theorem \ref{MAINTHEOREM1}. 
\begin{Remark}
Under the above conditions, $X(t)$ has only a density on $\Gamma(I)$, i.e.
the distribution may be singular on the set $A = \{ x \in \R_+^d \ | \ x_i = 0, \ \ i \in I \}$.
However, if one has $\P[X(t) \in \Gamma(I)] = 1$, then $\P[ X(t) \in A] = 0$ and hence $X(t)$ has a density on all $\R_+^d$.
Since the branching and diffusion mechanism vanishes at the boundary, 
one cannot avoid  to study the boundary behavior of multi-type CBI processes.
For results applicable to one-dimensional processes we refer to \cite{CPU13}, \cite{DFM14} and \cite{FU14}, see also the references therein.
It is also possible to obtain sufficient conditions for $\P[X(t) \in \Gamma(I)] = 1$ in arbitrary dimension;
this will be studied in a seperate work.
\end{Remark}
Note that the particular choice $\gamma_{0}=2$ is always possible, in which case Theorem \ref{MAINTHEOREM1}
is precisely Theorem \ref{MAINTHEOREM}. For $\gamma_{0}<\frac{1+\sqrt{5}}{2}$,
one has $\frac{\gamma_{0}^{2}}{1+\gamma_{0}}<1$ and hence we may
take $\alpha_{i}\in(\frac{\gamma_{0}^{2}}{1+\gamma_{0}},1)$. In this
case smoothing by immigration (see Example \ref{EXAMPLE:00}.(c)) may occur, which gives the following corollary.
\begin{Corollary}
Let $X$ be a multi-type CBI process and suppose that
the same conditions as in Theorem \ref{MAINTHEOREM1} are satisfied.
If $X(0)$ satisfies $\E[|X(0)|^{1 + \tau}] < \infty$, then 
\[
 \P[X_{i}(t)=0,\ \ i\not\in I]=0, \ \ t > 0.
\]
\end{Corollary} 
\begin{proof}
 The set $\{ x \in \R_+^d \ | \ x_i = 0, \ \ i \not \in I\} \subset \Gamma(I)$ has
 Lebesgue measure zero. Since $X(t)$ has a density on $\Gamma(I)$, the assertion is proved.
\end{proof}
Note that this corollary is only applicable in the presence of jumps from the
immigration. Indeed, if $\nu=0$, then condition (A) can be only satisfied for $I=\{1,\dots,d\}$.
Another sufficient condition for $\P[ X_i(t) = 0, \ \ i \in \{1,\dots, d\}] = 0$, $t > 0$, will be discussed in a seperate work.

Let us finally consider a particular case without diffusion where
the measures $\mu_{1},\dots,\mu_{d}$ have the specific form
\begin{align}
\mu_{k}(dz)=\widetilde{\mu}_{k}(dz_{k})\otimes\prod\limits _{j\neq k}\delta_{0}(dz_{j}),\ \ k\in\{1,\dots,d\},\label{EQ:10}
\end{align}
with $\widetilde{\mu}_{k}$ being Lévy measures on $\R_{+}$ satisfying
$\widetilde{\mu}_{k}(\{0\})=0$. In this case we obtain the following
analogue of our previous statements. 
\begin{Theorem}\label{MAINTHEOREM2}
Let $X$ be a multi-type CBI process with admissible parameters $(c,\beta,B,\nu,\mu)$ and assume that \eqref{EQ:10}
holds and that $c_{1}=\dots=c_{d}=0$. Moreover suppose that
\begin{enumerate}
\item[(a)] Condition (A) is satisfied for some $I\subset\{1,\dots,d\}$ and
$\alpha_{1},\dots,\alpha_{d}\in(0,2)$.
\item[(b)] For each $j\in\{1,\dots,d\}$ there exists $\gamma_{0}^{j}\in(1,2]$
and $\tau_{j}\in(0,\gamma_{0}^{j}-1)$ such that
\[
\int\limits _{\R_{+}}\left(z^{\gamma_{0}^{j}}\1_{\{z\leq1\}}+z^{1+\tau_{j}}\1_{\{z>1\}}\right)\widetilde{\mu}_{j}(dz)+\int\limits _{|z|>1}|z|^{1+\tau_{j}}\nu(dz)<\infty.
\]
If $I=\emptyset$, then we may also take $\tau_{1}=\dots=\tau_{d}=0$.
\item[(c)] It holds that $\alpha_{i}>\frac{\max\{\gamma_{0}^{1},\dots,\gamma_{0}^{d}\}}{1+\max\{\gamma_{0}^{1},\dots,\gamma_{0}^{d}\}}\gamma_{0}^{i}$.
Moreover, for each $j\in I$, one has $\alpha_{j}\geq1$.
\end{enumerate}
If $X(0)$ satisfies $\E[|X(0)|^{1 + \tau}] < \infty$, then for each $t>0$, 
$X(t)$ has a density $g_{t}$ on $\Gamma(I)$. Moreover, $f_{t}(x):=\rho_{I}(x)g_{t}(x)\in B_{1,\infty}^{\lambda,a}(\R^{d})$,
where $\lambda>0$ is small enough and $a$ is defined in \eqref{APPROX:03}.
\end{Theorem} Note that we have not assumed anything for the drift
component $B$. In some particular cases where $B$ does not mix different
components too much, it is possible to obtain results
with less restrictions on the parameters $\alpha_{i},\gamma_{0}^{i}$,
etc. It is possible, but would be awful, to formulate a general statement.
It is more convenient to apply the methods of this work directly to
particular models of this type.

\section{Main ingredients in the proofs}

\subsection{Anisotropic integration by parts}

Define the anisotropic Hölder-Zygmund space $C_{b}^{\lambda,a}(\R^{d})$
as the Banach space of functions $\phi$ with finite norm
\[
\|\phi\|_{C_{b}^{\lambda,a}}=\|\phi\|_{\infty}+\sum\limits _{k=1}^{d}\sup\limits _{h\in[-1,1]}|h|^{-\lambda/a_{k}}\|\Delta_{he_{k}}\phi\|_{\infty}.
\]
The following is our main technical tool for the existence of a density.
\begin{Lemma}\label{LEMMA:02} Let $a=(a_{1},\dots,a_{d})$ be an
anisotropy in the sense of \eqref{EQ:35} and $\lambda,\eta>0$ be
such that $(\lambda+\eta)/a_{k}\in(0,1)$ holds for all $k=1,\dots,d$.
Suppose that $q$ is a finite measure over $\R^{d}$ and there exists
$A>0$ such that, for all $\phi\in C_{b}^{\eta,a}(\R^{d})$ and all
$k=1,\dots,d$,
\begin{align}
\left|\int\limits _{\R^{d}}(\phi(x+he_{k})-\phi(x))q(dx)\right|\leq A\|\phi\|_{C_{b}^{\eta,a}}|h|^{(\lambda+\eta)/a_{k}},\ \ \forall h\in[-1,1].\label{EQ:32}
\end{align}
Then $q$ has a density $g$ with respect to the Lebesgue measure
such that
\begin{align*}
\|g\|_{B_{1,\infty}^{\lambda,a}}\leq q(\R^{d})+3dA(2d)^{\eta/\lambda}\left(1+\frac{\lambda}{\eta}\right)^{1+\frac{\eta}{\lambda}}.
\end{align*}
\end{Lemma} A proof of this Lemma is given in \cite{FJR18}. The
isotropic case, i.e. $a_{1}=\dots=a_{d}=1$, was first given in \cite{DF13,DR14,F15}.
Note that the restriction $(\lambda+\eta)/a_{k}\in(0,1)$, $k=1,\dots,d$,
is not essential since we may always replace $\lambda,\eta>0$ by
some smaller values which satisfy this condition and \eqref{EQ:32}.

\subsection{Multi-type CBI processes as strong solutions to stochastic equations}

Our proof relies on the representation of multi-type CBI processes
as solutions to a stochastic differential equation which is described
below. Let $(c,\beta,B,\nu,\mu)$ be a tuple of admissible parameters and $(\Omega,\F,\P)$ be a complete probability space
rich enough to support following objects
\begin{enumerate}
\item[(i)] A $d$-dimensional Brownian motion $W=(W(t))_{t\geq0}$.
\item[(ii)] Poisson random measures $N_{1},\dots,N_{d}$ on $\R_{+}\times\R_{+}^{d}\times\R_{+}$
with compensators
\[
\widehat{N}_{j}(du,dz,dr)=du\mu_{j}(dz)dr,\ \ j\in\{1,\dots,d\}.
\]
\item[(iii)] A Poisson random measure $N_{\nu}$ on $\R_{+}\times\R_{+}^{d}$
with compensator $\widehat{N}_{\nu}(ds,dz)=ds\nu(dz)$.
\end{enumerate}
The objects $W,N_{\nu},N_{1},\dots,N_{d}$ are supposed to be mutually
independent. Denote by $\widetilde{N}_{j}=N_{j}-\widehat{N}_{j}$,
$j\in\{1,\dots,d\}$, and $\widetilde{N}_{\nu}=N_{\nu}-\widehat{N}_{\nu}$
the corresponding compensated Poisson random measures. Let $X(0)$
be a random variable independent of the noise $W,N_{\nu},N_{1},\dots,N_{d}$.
Then
\begin{align}
X(t) & =X(0)+\int\limits _{0}^{t}\left(\beta+\widetilde{B}X(s)\right)ds+\sum\limits _{k=1}^{d}\sqrt{2c_{k}}e_{k}\int\limits _{0}^{t}\sqrt{X_{k}(s)}dW_{k}(s)+\int\limits _{0}^{t}\int\limits _{\R_{+}^{d}}zN_{\nu}(ds,dz)\label{SDE:CBI}\\
 & \ \ \ +\sum\limits _{j=1}^{d}\int\limits _{0}^{t}\int\limits _{|z|\leq1}\int\limits _{\R_{+}}z\1_{\{r\leq X_{j}(s-)\}}\widetilde{N}_{j}(ds,dz,dr)+\sum\limits _{j=1}^{d}\int\limits _{0}^{t}\int\limits _{|z|>1}\int\limits _{\R_{+}}z\1_{\{r\leq X_{j}(s-)\}}N_{j}(ds,dz,dr)\nonumber
\end{align}
has a pathwise unique strong solution, see \cite{BLP15}. Here $\widetilde{B}=(\widetilde{b}_{ij})_{i,j\in\{1,\dots,d\}}$
is obtained by changing the compensator of the jump operator involving
$(\mu_1,\dots, \mu_d)$. It is given by
\[
\widetilde{b}_{ij}=b_{ij}+\1_{\{i\neq j\}}\int\limits _{|z|\leq1}z_{i}\mu_{j}(dz)-\1_{\{i=j\}}\mu_{i}(\{|z|>1\}),\ \ i,j\in\{1,\dots,d\}.
\]
Note that $\widetilde{B}$ is well-defined and has non-negative off-diagonal
entries. An application of the It\^{o}-formula
shows that $X$ solves the martingale problem with generator \eqref{GENERATOR},
i.e. is a multi-type CBI process with admissible parameters $(c,\beta,B,\nu,\mu)$.

\subsection{Structure of the work}

This work is organized as follows. In Section 4 we provide a general
statement on the existence of densities for solutions to stochastic
equations with Hölder continuous coefficients driven by a Brownian
motion and a Poisson random measure. Our main results for multi-type
CBI processes are then deduced in Section 5 from the results obtained in Section 4. 
Section 6 is devoted to the discussion of sufficient conditions for (A), 
while particular examples illustrating how our main results from Section 2 can be applied are discussed in Section 7.
Some technical estimates for stochastic integrals with respect to Poisson random measures are collected in the appendix.

\section{A general criterion for existence of a density}

\subsection{Description of the model}

In this section we prove a general statement applicable to a wide
class of stochastic equations driven by Brownian motions and Poisson random measures. 
Such equations should, in particular, include \eqref{SDE:CBI}.
Motivated by multi-type CBI processes we consider unbounded coefficients
and treat the case of compensated small jumps, jumps of finite variation and big jumps separately.

Let $(\Omega,\F,(\F_{t})_{t\geq0},\P)$ be a stochastic basis with
the usual conditions, i.e. $(\Omega,\F,\P)$ is complete, $\F_{0}$
contains all $\P$-null sets and $(\F_{t})_{t\geq0}$ is a right-continuous
filtration over $\F$. Suppose that the stochastic basis is rich enough
to support the following objects
\begin{enumerate}
\item[(i)] A $d$-dimensional $(\F_{t})_{t\geq0}$-Brownian motion $W=(W(t))_{t\geq0}$.
\item[(ii)] An $(\F_{t})_{t\geq0}$-Poisson random measure $N$ with compensator
$\widehat{N}(du,dz)=dum(dz)$ on $\R_{+}\times E$, where $m$ is
a $\sigma$-finite measure on some Polish space $E$.
\end{enumerate}
Both terms are supposed to be independent. Denote by $\widetilde{N}=N-\widehat{N}$
the corresponding compensated Poisson random measure. Let $X(0)$
be an $\F_{0}$-measurable random variable independent of $W$ and
$N$. Consider an $(\F_{t})_{t\geq0}$-adapted cádlág-process $X=(X(t))_{t\geq0}$
satisfying
\begin{align}
X(t) & =X(0)+\int\limits _{0}^{t}b(X(u))du+\int\limits _{0}^{t}\sigma(X(t))dW(t)+\int\limits _{0}^{t}\int\limits _{E_{0}}\sigma^{0}(X({u-}),z)\widetilde{N}(du,dz)\label{EQ:03}\\
 & \ \ \ +\int\limits _{0}^{t}\int\limits _{E_{1}}\sigma^{1}(X({u-}),z)N(du,dz)+\int\limits _{0}^{t}\int\limits _{E_{2}}\sigma^{2}(X({u-}),z)N(du,dz),\nonumber
\end{align}
where $E=E_{0}\cup E_{1}\cup E_{2}$ and $E_{0},E_{1},E_{2}$ are
disjoint sets with $m(E_{2})<\infty$. We suppose that $b,\sigma:\R^{d}\longrightarrow\R^{d}$,
$\sigma^{0},\sigma^{1},\sigma^{2}:\R^{d}\times E\longrightarrow\R^{d}$
are measurable, and satisfy
\begin{align*}
\sup\limits _{|x|\leq R}\left(|b(x)|+|\sigma(x)|+\int\limits _{E_{0}}|\sigma^{0}(x,z)|^{2}m(dz)+\int\limits _{E_{1}}|\sigma^{1}(x,z)|m(dz)\right)<\infty,\ \ R>0.
\end{align*}
This implies, in particular, that the corresponding stochastic integrals
in \eqref{EQ:03} are well-defined. Here $E_{0}$ corresponds to small
(compensated) jumps, $E_{1}$ to jumps of finite variation and $E_{2}$
to big jumps. \begin{Remark}
\begin{enumerate}
\item[(i)] One typically absorbs the finite variation terms into the definition
of $\sigma^{0},\sigma^{2}$, i.e., one has $E^{1}=\emptyset$ and
$\sigma^{1}=0$. However, having applications in mind it is reasonable
to treat this cases differently.
\item[(ii)] At first one may think that \eqref{SDE:CBI} is more general, since
it contains different independent Poisson random measures. However,
since the particular form of $E$ is not specified, we also cover
this case as it is shown in Section 5.
\item[(iii)] It is straightforward to extend all results obtained below to time-dependent
coefficients.
\end{enumerate}
\end{Remark}

\subsection{Hölder regularity in time}

Motivated by \eqref{SDE:CBI}, we suppose that the coefficients of
\eqref{EQ:03} are Hölder continuous and not necessarily bounded.
Since an unbounded function $f$ might be Hölder continuous with exponent
$\gamma\in(0,1]$ without being Hölder continuous with exponent $\gamma'\in(0,\gamma)$,
we have to keep track of the Hölder continuity for each component
separately, see also Section 7 for particular examples. Below we suppose that the following conditions are satisfied:
\begin{enumerate}
\item[(B1)] For each $i\in\{1,\dots,d\}$, there exist $J_{i}(b)\subset\{1,\dots,d\}$,
$\theta_{i}(b)\in[0,1]$ and $C>0$ such that
\begin{align*}
|b_{i}(x)-b_{i}(y)|\leq C\sum\limits _{j\in J_{i}(b)}|x_{j}-y_{j}|^{\theta_{i}(b)}.
\end{align*}
\item[(B2)] For each $i\in\{1,\dots,d\}$, there exist $J_{i}(\sigma^{0}),J_{i}(\sigma^{1}),J_{i}(\sigma^{2})\subset\{1,\dots,d\}$,
$\theta_{i}(\sigma^{0}),\theta_{i}(\sigma^{1}),\theta_{i}(\sigma^{2})\in[0,1]$,
$\gamma_{i}(\sigma^{0})\in(1,2]$, $\gamma_{i}(\sigma^{1})\in(0,1]$,
$\gamma_{i}(\sigma^{2})\in(0,\gamma_{i}(\sigma^{0})]$ and $C>0$
such that
\[
\int\limits _{E_{k}}|\sigma_{i}^{k}(x,z)-\sigma_{i}^{k}(y,z)|^{\gamma_{i}(\sigma^{k})}m(dz)\leq C\sum\limits _{j\in J_{i}(\sigma^{k})}|x_{j}-y_{j}|^{\theta_{i}(\sigma^{k})\gamma_{i}(\sigma^{k})},\ \ \ k\in\{0,1,2\}.
\]
\item[(B3)] For each $i\in\{1,\dots,d\}$, there exists $J_{i}(\sigma)\subset\{1,\dots,d\}$
and $\theta_{i}(\sigma)\in[0,1]$ such that
\[
|\sigma_{ik}(x)-\sigma_{ik}(y)|\leq\sum\limits _{j\in J_{i}(\sigma)}|x_{j}-y_{j}|^{\theta_{i}(\sigma)},\ \ k\in\{1,\dots,d\}.
\]
\end{enumerate}
Thus $(\theta_{i}(b),\theta_{i}(\sigma),\theta_{i}(\sigma^{0}),\theta_{i}(\sigma^{1}),\theta_{i}(\sigma^{2}))$,
$i\in\{1,\dots,d\}$, describe the Hölder exponents for the coefficients
with respect to the space variables while the coupling of different
components is described by the sets $J_{i}(b),J_{i}(\sigma),J_{i}(\sigma^{0}),J_{i}(\sigma^{1}),J_{i}(\sigma^{2})$,
$i\in\{1,\dots,d\}$. These sets are motivated by the particular form
of \eqref{SDE:CBI}. Define
\[
\gamma_{i}=\max\{\1_{\sigma_{i}\neq0}2,\1_{\sigma_{i}^{0}\neq0}\gamma_{i}(\sigma^{0}),\1_{\sigma_{i}^{1}\neq0}\gamma_{i}(\sigma^{1}),\1_{\sigma_{i}^{2}\neq0}\gamma_{i}(\sigma^{2})\},
\]
where $\sigma_{i}=(\sigma_{i1},\dots,\sigma_{id})$, and similarly
let
\[
\gamma_{*,i}=\min\{\1_{\sigma_{i}\neq0}2,\1_{\sigma_{i}^{0}\neq0}\gamma_{i}(\sigma^{0}),\1_{\sigma_{i}^{1}\neq0}\gamma_{i}(\sigma^{1}),\1_{\sigma_{i}^{2}\neq0}\gamma_{i}(\sigma^{2})\}.
\]
We start with an estimate on time Hölder regularity for processes
$X$ given as in \eqref{EQ:03}. \begin{Lemma}\label{LEMMA:01} Suppose
that (B1) – (B3) are satisfied, fix $i\in\{1,\dots,d\}$ and let $\eta\in(0,\gamma_{*,i}]$.
Then, there exists a constant $C>0$ such that, for all $0\leq s\leq t\leq s+1$
and any $X$ as in \eqref{EQ:03}, one has
\[
\E[|X_{i}(t)-X_{i}(s)|^{\eta}]\leq C(t-s)^{\frac{\eta}{\gamma_{i}}}M_{i}(t,\eta),
\]
where the constant $C$ is independent of $X$, and
\begin{align*}
M_{i}(t,\eta) & =\sum\limits _{k=1}^{d}\sup\limits _{u\in[s,t]}\E[|\sigma_{ik}(X(u))|^{2}]^{\eta/2}+\sup\limits _{u\in[0,t]}\begin{cases}
\E[|b_{i}(X(u))|^{\eta}], & \eta\geq1\\
\E[|b_{i}(X(u))|]^{\eta}, & \eta\in(0,1)
\end{cases}\\
 & \ \ \ +\sup\limits _{u\in[0,t]}\E\left[\int\limits _{E_{0}}|\sigma_{i}^{0}(X(u),z)|^{\gamma_{i}(\sigma^{0})}m(dz)\right]^{\eta/\gamma_{i}(\sigma^{0})}\\
 & \ \ \ +\sup\limits _{u\in[0,t]}\E\left[\int\limits _{E_{1}}|\sigma_{i}^{1}(X(u),z)|^{\gamma_{i}(\sigma^{1})}m(dz)\right]^{\eta/\gamma_{i}(\sigma^{1})}\\
 & \ \ \ +\sup\limits _{u\in[0,t]}\E\left[\int\limits _{E_{2}}|\sigma_{i}^{2}(X(u),z)|^{\gamma_{i}(\sigma^{2})}m(dz)\right]^{\eta/\gamma_{i}(\sigma^{2})}.
\end{align*}
\end{Lemma} \begin{proof} Observe that
\begin{align*}
 & \E[|X_{i}(t)-X_{i}(s)|^{\eta}]\leq C\E\left[\left|\int\limits _{s}^{t}b_{i}(X(u))du\right|^{\eta}\right]+C\sum\limits _{k=1}^{d}\E\left[\left|\int\limits _{s}^{t}\sigma_{ik}(X(u))dW_{k}(u)\right|^{\eta}\right]\\
 & \ \ \ +C\E\left[\left|\int\limits _{s}^{t}\int\limits _{E_{0}}\sigma_{i}^{0}(X(u-),z)\widetilde{N}(du,dz)\right|^{\eta}\right]\\
 & \ \ \ +C\E\left[\left|\int\limits _{s}^{t}\int\limits _{E_{1}}\sigma_{i}^{1}(X(u-),z)N(du,dz)\right|^{\eta}\right]+C\E\left[\left|\int\limits _{s}^{t}\int\limits _{E_{2}}\sigma_{i}^{2}(X(u-),z)N(du,dz)\right|^{\eta}\right].
\end{align*}
The first term is, for $\eta\geq1$, estimated by the Hölder inequality
\[
\E\left[\left|\int\limits _{s}^{t}b_{i}(X(u))du\right|^{\eta}\right]\leq C(t-s)^{\eta}\sup\limits _{u\in[0,t]}\E\left[|b_{i}(X(u))|^{\eta}\right]
\]
and for $\eta\in(0,1)$ we get
\[
\E\left[\left|\int\limits _{s}^{t}b_{i}(X(u))du\right|^{\eta}\right]\leq\E\left[\left|\int\limits _{s}^{t}b_{i}(X(u))du\right|\right]^{\eta}\leq(t-s)^{\eta}\sup\limits _{u\in[0,t]}\E\left[|b_{i}(X(u))|\right]^{\eta}.
\]
For the stochastic integral with respect to the Brownian motion we
obtain from the BDG-inequality
\begin{align*}
\E\left[\left|\int\limits _{s}^{t}\sigma_{ik}(X(u))dW_{k}(u)\right|^{\eta}\right] & \leq\E\left[\left|\int\limits _{s}^{t}|\sigma_{ik}(X(u))|^{2}du\right|^{\eta/2}\right]\\
 & \leq(t-s)^{\eta/2}\sup\limits _{u\in[s,t]}\E[|\sigma_{ik}(X(u))|^{2}]^{\eta/2}.
\end{align*}
For the second stochastic integral we get by Lemma \ref{LEMMA:00}.(a)
immediately
\[
\E\left[\left|\int\limits _{s}^{t}\int\limits _{E_{0}}\sigma_{i}^{0}(X(u-),z)\widetilde{N}(du,dz)\right|^{\eta}\right]\leq C(t-s)^{\eta/\gamma_{i}(\sigma^{0})}\sup\limits _{u\in[0,t]}\E\left[\int\limits _{E_{0}}|\sigma_{i}^{0}(X(u),z)|^{\gamma_{i}(\sigma^{0})}m(dz)\right]^{\eta/\gamma_{i}(\sigma^{0})}.
\]
For the third stochastic integral we apply Lemma \ref{LEMMA:00}.(b)
so that
\begin{align*}
\E\left[\left|\int\limits _{s}^{t}\int\limits _{E_{1}}\sigma_{i}^{1}(X(u-),z)N(du,dz)\right|^{\eta}\right] & \leq C(t-s)^{\eta/\gamma_{i}(\sigma^{1})}\sup\limits _{u\in[0,t]}\E\left[\int\limits _{E_{1}}|\sigma_{i}^{1}(X(u),z)|^{\gamma_{i}(\sigma^{1})}m(dz)\right]^{\eta/\gamma_{i}(\sigma^{1})}.
\end{align*}
For the last stochastic integral we obtain by Lemma \ref{LEMMA:00}.(b)
and $m(E_{2})<\infty$
\begin{align*}
\E\left[\left|\int\limits _{s}^{t}\int\limits _{E_{2}}\sigma_{i}^{2}(X(u-),z)N(du,dz)\right|^{\eta}\right]\leq C(t-s)^{\eta/\gamma_{i}(\sigma^{2})}\sup\limits _{u\in[0,t]}\E\left[\int\limits _{E_{2}}|\sigma_{i}^{2}(X(u),z)|^{\gamma_{i}(\sigma^{2})}m(dz)\right]^{\eta/\gamma_{i}(\sigma^{2})}{\color{red}.}
\end{align*}
The assertion now follows from $t-s\leq1$ and since $\gamma_{i}(\sigma^{0}),\gamma_{i}(\sigma^{1}),\gamma_{i}(\sigma^{2}),2\leq\gamma_{i}$.
\end{proof}

\subsection{The approximation}

For coefficients $b,\sigma,\sigma^{0},\sigma^{1},\sigma^{2}$ satisfying
(B1) – (B3), $i\in\{1,\dots,d\}$ and $k\in\{0,1,2\}$, define
\begin{align*}
\kappa(\sigma_{i}^{k}) & =\begin{cases}
\frac{1}{\gamma_{i}(\sigma^{k})}+\frac{\theta_{i}(\sigma^{k})}{\gamma_{i}^{*}}, & \sigma_{i}^{k}\text{ is not constant }\\
+\infty, & \sigma_{i}^{k}\text{ is constant }
\end{cases},\\
\kappa(\sigma_{i}) & =\begin{cases}
\frac{1}{2}+\frac{\theta_{i}(\sigma)}{\gamma_{i}^{*}}, & \sigma_{i}\text{ is not constant }\\
+\infty, & \sigma_{i}\text{ is constant }
\end{cases},\\
\kappa(b_{i}) & =\begin{cases}
1+\frac{\theta_{i}(b)}{\gamma_{i}^{*}}, & b_{i}\text{ is not constant }\\
+\infty, & b_{i}\text{ is constant }
\end{cases},
\end{align*}
where $\sigma_{i}=(\sigma_{i1},\dots,\sigma_{1d})$ and $\gamma_{i}^{*}=\max\{\gamma_{j}\ |\ j\in J_{i}(b)\cup J_{i}(\sigma)\cup J_{i}(\sigma^{k}),\ k\in\{0,1,2\}\}$.
Define
\begin{align*}
\kappa_{i} & =\min\left\{ \kappa(\sigma_{i}),\kappa(\sigma_{i}^{0}),\kappa(\sigma_{i}^{1}),\kappa(\sigma_{i}^{2}),\kappa(b_{i})\right\} .
\end{align*}
Hence, in cases where some of the coefficients $b_{i},\sigma_{i},\sigma_{i}^{0},\sigma_{i}^{1},\sigma_{i}^{2}$
are constant, we omit the corresponding terms in the definition of
$\kappa_{i}$ and set $J_{i}(b)=\emptyset$ or $J_{i}(\sigma)=\emptyset$
or $J_{i}(\sigma_{i}^{k})=\emptyset$, respectively. The following
is the main estimate for this section. \begin{Proposition}\label{PROP:00}
Suppose that (B1) – (B3) are satisfied and fix $i\in\{1,\dots,d\}$.
Moreover, suppose that, for each $j\in\{1,\dots,d\}$,
\begin{align}
\max\{\1_{J_{i}(b)}(j)\theta_{i}(b),\1_{J_{i}(\sigma)}(j)2\theta_{i}(\sigma),\1_{J_{i}(\sigma^{k})}(j)\theta_{i}(\sigma^{k})\gamma_{i}(\sigma^{k}),\}\leq\gamma_{*,j},\ \ k\in\{0,1,2\}.\label{ADMISSIBLE}
\end{align}
Let $X$ be as in \eqref{EQ:03} and define, for $t>0$ and $\e\in(0,1\wedge t]$,
the approximation $X_{i}^{\e}(t)=U_{i}^{\e}(t)+V_{i}^{\e}(t)$, where
\begin{align}
U_{i}^{\e}(t) & =X_{i}(t-\e)+\e b_{i}(X(t-\e))+\int\limits _{t-\e}^{t}\int\limits _{E_{2}}\sigma_{i}^{2}(X(t-\e),z)N(du,dz),\label{APPROX:00}\\
V_{i}^{\e}(t) & =\sum\limits _{k=1}^{d}\sigma_{ik}(X(t-\e))(W_{k}(t)-W_{k}(t-\e))\label{APPROX:01}\\
 & \ \ \ +\int\limits _{t-\e}^{t}\int\limits _{E_{0}}\sigma_{i}^{0}(X(t-\e),z)\widetilde{N}(du,dz)+\int\limits _{t-\e}^{t}\int\limits _{E_{1}}\sigma_{i}^{1}(X(t-\e),z)N(du,dz).\nonumber
\end{align}
Then, for any $0<\eta\leq1\wedge\gamma_{*,i}$,
\[
\E[|X_{i}(t)-X_{i}^{\e}(t)|^{\eta}]\leq C\e^{\eta\kappa_{i}}H_{i}(t,\eta),\ \ t > 0, \ \ \e\in(0,1\wedge t],
\]
where the constant $C > 0$ is independent of $\e$, $t$ and $X$, and
\begin{align*}
H_{i}(t,\eta) & =\sum\limits _{j\in J_{i}(\sigma^{0})}M_{j}(t,\theta_{i}(\sigma^{0})\gamma_{i}(\sigma^{0}))^{\eta/\gamma_{i}(\sigma^{0})}+\sum\limits _{j\in J_{i}(\sigma^{1})}M_{j}(t,\theta_{i}(\sigma^{1})\gamma_{i}(\sigma^{1}))^{\eta/\gamma_{i}(\sigma^{1})}\\
 & \ \ \ +\sum\limits _{j\in J_{i}(\sigma^{2})}M_{j}(t,\theta_{i}(\sigma^{2})\gamma_{i}(\sigma^{2}))^{\eta/\gamma_{i}(\sigma^{2})}+\sum\limits _{j\in J_{i}(b)}M_{j}(t,\theta_{i}(b))^{\eta}+\sum\limits _{j\in J_{i}(\sigma)}M_{j}(t,2\theta_{i}(\sigma))^{\eta/2}.
\end{align*}
\end{Proposition} \begin{proof} Fix $t>0$, $\e\in(0,1\wedge t]$
and let $\eta\in(0,1\wedge\gamma_{*,i}]$. Then
\begin{align*}
 & \ \E[|X_{i}(t)-X_{i}^{\e}(t)|^{\eta}]\leq R_{0}+R_{1}+R_{2}+R_{3}+R_{4},\\
R_{0} & =\E\left[\left|\int\limits _{t-\e}^{t}\int\limits _{E_{0}}\left(\sigma_{i}^{0}(X(u-),z)-\sigma_{i}^{0}(X(t-\e),z)\right)\widetilde{N}(du,dz)\right|^{\eta}\right],\\
R_{1} & =\E\left[\left|\int\limits _{t-\e}^{t}\int\limits _{E_{1}}\left(\sigma_{i}^{1}(X(u-),z)-\sigma_{i}^{1}(X(t-\e),z)\right)N(du,dz)\right|^{\eta}\right],\\
R_{2} & =\E\left[\left|\int\limits _{t-\e}^{t}\int\limits _{E_{2}}\left(\sigma_{i}^{2}(X(u-),z)-\sigma_{i}^{2}(X(t-\e),z)\right)N(du,dz)\right|^{\eta}\right],\\
R_{3} & =\E\left[\left|\int\limits _{t-\e}^{t}(b_{i}(X(u))-b_{i}(X(t-\e)))du\right|^{\eta}\right],\\
R_{4} & =\sum\limits _{k=1}^{d}\E\left[\left|\int\limits _{t-\e}^{t}\left(\sigma_{ik}(X(u))-\sigma_{ik}(X(t-\e))\right)dW_{k}(u)\right|^{\eta}\right].
\end{align*}
For $R_{0}$ we first apply Lemma \ref{LEMMA:00}, then condition
(B2) and finally Lemma \ref{LEMMA:01} to obtain
\begin{align*}
R_{0} & \leq C\e^{\eta/\gamma_{i}(\sigma^{0})}\sup\limits _{u\in(t-\e,t]}\E\left[\int\limits _{E_{0}}\left|\sigma_{i}^{0}(X(u-),z)-\sigma_{i}^{0}(X(t-\e),z)\right|^{\gamma_{i}(\sigma^{0})}m(dz)\right]^{\eta/\gamma_{i}(\sigma^{0})}\\
 & \leq C\e^{\eta/\gamma_{i}(\sigma^{0})}\sum\limits _{j\in J_{i}(\sigma^{0})}\sup\limits _{u\in(t-\e,t]}\E\left[|X_{j}(u-)-X_{j}(t-\e)|^{\theta_{i}(\sigma^{0})\gamma_{i}(\sigma^{0})}\right]^{\eta/\gamma_{i}(\sigma^{0})}\\
 & \leq C\e^{\eta/\gamma_{i}(\sigma^{0})}\sum\limits _{j\in J_{i}(\sigma^{0})}\e^{\eta\theta_{i}(\sigma^{0})/\gamma_{j}}M_{j}(t,\theta_{i}(\sigma^{0})\gamma_{i}(\sigma^{0}))^{\eta/\gamma_{i}(\sigma^{0})}.
\end{align*}
For $R_{1}$ we apply Lemma \ref{LEMMA:00}.(b), use assumption (B2)
and proceed as before to deduce
\begin{align*}
R_{1}\leq\e^{\eta/\gamma_{i}(\sigma^{1})}\sum\limits _{j\in J^{i}(\sigma^{1})}\e^{\eta\theta_{i}(\sigma^{1})/\gamma_{j}}M_{j}(t,\theta_{i}(\sigma^{1})\gamma_{i}(\sigma^{1}))^{\eta/\gamma_{i}(\sigma^{1})}.
\end{align*}
For $R_{2}$ we apply Lemma \ref{LEMMA:00}.(b) and proceed as before
to deduce
\begin{align*}
R_{2}\leq\e^{\eta/\gamma_{i}(\sigma^{2})}\sum\limits _{j\in J^{i}(\sigma^{2})}\e^{\eta\theta_{i}(\sigma^{2})/\gamma_{j}}M_{j}(t,\theta_{i}(\sigma^{2})\gamma_{i}(\sigma^{2}))^{\eta/\gamma_{i}(\sigma^{2})}.
\end{align*}
For $R_{3}$ we apply (B1) and Lemma \ref{LEMMA:01} so that
\begin{align*}
R_{3} & \leq C\E\left[\sum\limits _{j\in J_{i}(b)}\int\limits _{t-\e}^{t}|X_{j}(u-)-X_{j}(t-\e)|^{\theta_{i}(b)}du\right]^{\eta}\\
 & \leq C\e^{\eta}\sum\limits _{j\in J_{i}(b)}\sup\limits _{u\in(t-\e,t]}\E[|X_{j}(u-)-X_{j}(t-\e)|^{\theta_{i}(b)}]^{\eta}\\
 & \leq C\e^{\eta}\sum\limits _{j\in J_{i}(b)}\e^{\eta\theta_{i}(b)/\gamma_{j}}M_{j}(t,\theta_{i}(b))^{\eta}.
\end{align*}
For the last term we obtain from the BDG-inequality, (B3) and Lemma
\ref{LEMMA:01}
\begin{align*}
R_{4} & \leq C\e^{\eta/2}\sup\limits _{u\in(t-\e,t]}\E\left[\left|\sigma_{ik}(X(u))-\sigma_{ik}(X(t-\e))\right|^{2}\right]^{\eta/2}\\
 & \leq C\e^{\eta/2}\sum\limits _{j\in J_{i}(\sigma)}\sup\limits _{u\in(t-\e,t]}\E\left[|X_{j}(u)-X_{j}(t-\e)|^{2\theta_{i}(\sigma)}\right]^{\eta/2}\\
 & \leq C\e^{\eta/2}\sum\limits _{j\in J_{i}(\sigma)}\e^{\eta\theta_{i}(\sigma)/\gamma_{j}}M_{j}(t,2\theta_{i}(\sigma))^{\eta/2}.
\end{align*}
Collecting all estimates and using the definition of $H_{i}(t,\eta)$
and $\kappa_{i}$ gives the assertion. \end{proof}

\subsection{Main estimate}

In this section we prove our main estimate used to prove existence
of densities. For each $t\in(0,1]$ and $x\in\R^{d}$ define
\begin{align}
L^{x}(t):=\sigma(x)W(t)+\int\limits _{0}^{t}\int\limits _{E_{0}}\sigma^{0}(x,z)\widetilde{N}(du,dz)+\int\limits _{0}^{t}\int\limits _{E_{1}}\sigma^{1}(x,z)N(du,dz).\label{EQ:07}
\end{align}
The following condition guarantees that the noise part has some smoothing property. As usual write $1/0:=+\infty$.
\begin{enumerate}
\item[(B4)] There exist $\rho:\R^{d}\longrightarrow[0,\infty)$ and $(\alpha_{i})_{i\in\{1,\dots,d\}}\subset(0,2]$
such that $L^{x}(t)$ has, for each $t\in(0,1]$ and $x\in\Gamma:=\{y\in\R^{d}\ |\ \rho(y)>0\}$,
a density $g_{t}^{x}$ with respect to the Lebesgue measure and, for
all $i\in\{1,\dots,d\}$,
\begin{align}
\limsup\limits _{t\to0}t^{1/\alpha_{i}}\int\limits _{\R^{d}}\left|g_{t}^{x}(z+e_{i}h)-g_{t}^{x}(z)\right|dz\leq\frac{|h|}{\rho(x)},\ \ h\in[-1,1].\label{EQ:33}
\end{align}
\end{enumerate}
Note that, for $x\not\in\Gamma$, the right-hand side of \eqref{EQ:33}
equals to $+\infty$ in which case nothing has to be verified. 
The following is our main estimate for this section. 
\begin{Proposition}\label{PROP:01}
Assume that (B1) – (B4) are satisfied and suppose that \eqref{ADMISSIBLE}
holds for all $i\in\{1,\dots,d\}$. Let $(X(t))_{t\geq0}$ be as in
\eqref{EQ:03} with the additional properties
\begin{enumerate}
\item[(i)] There exists $\tau>0$ such that, for each $i\in\{1,\dots,d\}$,
one has
\[
G_{j,i}(t)<\infty,\qquad j\in J_{i}(b)\cup J_{i}(\sigma)\cup J_{i}(\sigma^{0})\cup J_{i}(\sigma^{1})\cup J_{i}(\sigma^{2}),
\]
where $\zeta_{i}=\max\{1,2\theta_{i}(\sigma),\theta_{i}(\sigma^{0})\gamma_{i}(\sigma^{0}),\theta_{i}(\sigma^{1})\gamma_{i}(\sigma^{1}),\theta_{i}(\sigma^{2})\gamma_{i}(\sigma^{2})\}$
and
\begin{align*}
G_{j,i}(t) & =\sum\limits _{k=1}^{d}\sup\limits _{u\in[0,t]}\E[|\sigma_{jk}(X(u))|^{2}]+\sup\limits _{u\in[0,t]}\E\left[\int\limits _{E_{0}}|\sigma_{j}^{0}(X(u),z)|^{\gamma_{i}(\sigma^{0})}m(dz)\right]\\
 & \ \ \ +\sup\limits _{u\in[0,t]}\E\left[\int\limits _{E_{1}}|\sigma_{j}^{1}(X(u),z)|^{\gamma_{i}(\sigma^{1})}m(dz)\right]+\sup\limits _{u\in[0,t]}\E\left[\int\limits _{E_{2}}|\sigma_{j}^{2}(X(u),z)|^{\gamma_{i}(\sigma^{2})}m(dz)\right]\\
 & \ \ \ +\sup\limits _{u\in[0,t]}\E\left[|b_{j}(X(u))|^{\zeta_{i}}\right]+\sup\limits _{u\in[0,t]}\E[\rho(X(u))^{1+\tau}],
\end{align*}
\item[(ii)] There exists $\delta>0$ such that, for any $t>0$ and $\e\in(0,1\wedge t]$,
\begin{align}
\E[|\rho(X(t))-\rho(X(t-\e))|]\leq C\e^{\delta},\label{EQ:22}
\end{align}
where $C=C_{t}>0$ is independent of $\e$ and locally bounded in
$t$.
\end{enumerate}
Let $a=(a_{i})_{i\in\{1,\dots,d\}}$ be an anisotropy and $\eta\in(0,1)$
with
\begin{align}
\left(1+\frac{1}{\tau}\right)\frac{\eta}{a_{i}}\leq1\wedge\gamma_{*,i},\ \ i\in\{1,\dots,d\}.\label{EXAMPLE:25}
\end{align}
Then there exists a constant $C=C_{t,\eta}>0$ (locally bounded in
$t$) and $\e_{0}\in(0,1\wedge t)$ such that, for any $\e\in(0,\e_{0})$,
$h\in[-1,1]$, $\phi\in C_{b}^{\eta,a}(\R^{d})$ and $i\in\{1,\dots,d\}$,
\[
\left|\E\left[\rho(X(t))\Delta_{he_{i}}\phi(X(t))\right]\right|\leq C\|\phi\|_{C_{b}^{\eta,a}}\left(|h|^{\eta/a_{i}}\e^{\delta}+|h|\e^{-1/\alpha_{i}}+\max\limits _{j\in\{1,\dots,d\}}\e^{\eta\kappa_{j}/a_{j}}\right).
\]
\end{Proposition}
 \begin{proof} For $\e\in(0,1\wedge t)$ let $X^{\e}(t)$
be the approximation from Proposition \ref{PROP:00}. Then
\begin{align*}
\left|\E\left[\rho(X(t))\Delta_{he_{i}}\phi(X(t))\right]\right| & \leq R_{1}+R_{2}+R_{3},\\
R_{1} & =\left|\E\left[\Delta_{he_{i}}\phi(X(t))\left(\rho(X(t))-\rho(X(t-\e))\right)\right]\right|,\\
R_{2} & =\E\left[|\Delta_{he_{i}}\phi(X(t))-\Delta_{he_{i}}\phi(X^{\e}(t))|\rho(X(t-\e))\right],\\
R_{3} & =\left|\E\left[\rho(X(t-\e))\Delta_{he_{i}}\phi(X^{\e}(t))\right]\right|.
\end{align*}
For the first term we can use \eqref{EQ:22} to obtain
\begin{align*}
R_{1}\leq\|\phi\|_{C_{b}^{\eta,a}}|h|^{\eta/a_{i}}\E[|\rho(X(t))-\rho(X(t-\e))|]\leq C\|\phi\|_{C_{b}^{\eta,a}}|h|^{\eta/a_{i}}\e^{\delta}.
\end{align*}
For $R_{2}$, the Hölder inequality with $\frac{1}{1+\tau}+\frac{1}{1+\frac{1}{\tau}}=1$
implies
\begin{align*}
R_{2} & \leq C\|\phi\|_{C_{b}^{\eta,a}}\max\limits _{j\in\{1,\dots,d\}}\E\left[\rho(X(t-\e))|X_{j}(t)-X_{j}^{\e}(t)|^{\eta/a_{j}}\right]\\
 & \leq C\|\phi\|_{C_{b}^{\eta,a}}\sup\limits _{u\in[0,t]}\E\left[\rho(X(u))^{1+\tau}\right]^{1/(1+\tau)}\max\limits _{j\in\{1,\dots,d\}}\E\left[|X_{j}(t)-X_{j}^{\e}(t)|^{\left(1+\frac{1}{\tau}\right)\frac{\eta}{a_{j}}}\right]^{\frac{\tau}{1+\tau}}\\
 & \leq C\|\phi\|_{C_{b}^{\eta,a}}\sup\limits _{u\in[0,t]}\E\left[\rho(X(u))^{1+\tau}\right]^{1/(1+\tau)}\max\limits _{j\in\{1,\dots,d\}}\e^{\eta\kappa_{j}/a_{j}},
\end{align*}
where in the last inequality we have used \eqref{EXAMPLE:25} and
$G_{j,i}(t)<\infty$ so that Lemma \ref{LEMMA:01} is applicable.
Let us turn to $R_{3}$. Let $g_{t}^{x}$ be the density given by
(B4) and write $X^{\e}(t)=U^{\e}(t)+V^{\e}(t)$, where $U^{\e}(t)$
and $V^{\e}(t)$ are given by \eqref{APPROX:00} and \eqref{APPROX:01}.
By (B4) there exists $\e_{0}>0$ small enough such that for any $\e\in(0,\e_{0})$,
\begin{align*}
R_{3} & =\left|\E\left[\int\limits _{\R^{d}}\rho(X(t-\e))(\Delta_{he_{i}}\phi)(U^{\e}(t)+z)g_{\e}^{X(t-\e)}(z)dz\right]\right|\\
 & =\left|\E\left[\int\limits _{\R^{d}}\rho(X(t-\e))\phi(U^{\e}(t)+z)(\Delta_{-he_{i}}g_{\e}^{X(t-\e)})(z)dz\right]\right|\leq C\|\phi\|_{C_{b}^{\eta,a}}|h|\e^{-1/\alpha_{i}},
\end{align*}
where we have used \eqref{EQ:33}. Summing up the estimates for $R_{0},R_{1},R_{2},R_{3}$
yields the assertion. \end{proof} \begin{Remark} Note that for bounded
coefficients $b,\sigma,\sigma^{0},\sigma^{1},\sigma^{2}$ the restriction
$G_{j,i}(t)<\infty$ is automatically satisfied. More generally, in
many cases it suffices to show that $X$ has finite second moments.
For the particular case of multi-type CBI processes even less is sufficient,
see Section 2. \end{Remark}

\subsection{Existence of the density}

The following is the main result on the existence of densities for
\eqref{EQ:03}. \begin{Theorem}\label{TH:00} Assume that (B1) –
(B4) are satisfied and suppose that \eqref{ADMISSIBLE} holds for
all $i\in\{1,\dots,d\}$ and,
\begin{align}
\kappa_{i}\alpha_{i}>1,\ \ \forall i\in\{1,\dots,d\}.\label{EQ:21}
\end{align}
Let $(X(t))_{t\geq0}$ be as in \eqref{EQ:03} with the properties
(i) and (ii) from Proposition \ref{PROP:01}. Define an anisotropy
$a=(a_{i})_{i\in\{1,\dots,d\}}$ and a mean order of smoothness $\overline{\alpha}$
as in \eqref{APPROX:03}. Then there exists $\lambda\in(0,1)$ such
that the finite measure $q_{t}$ given by
\begin{align*}
q_{t}(A)=\E[\rho(X(t))\1_{A}(X(t))],\ \ \forall A\subset\R^{d}\ \ \text{ Borel },
\end{align*}
has, for every $t>0$, a density $g_{t}\in B_{1,\infty}^{\lambda,a}(\R^{d})$
with respect to the Lebesgue measure and
\begin{align*}
\|g_{t}\|_{B_{1,\infty}^{\lambda,a}}\leq q_{t}(\R^{d})+h(t)(1\wedge t)^{-1/\alpha^{\mathrm{min}}},
\end{align*}
where $h:[0,\infty)\longrightarrow(0,\infty)$ is locally bounded
in $t$ and $\alpha^{\mathrm{min}}=\min\{\alpha_{1},\dots,\alpha_{d}\}$.
\end{Theorem} \begin{proof} Let $t>0$ be fixed. It suffices to
show that Lemma \ref{LEMMA:02} is applicable to $q_{t}$. 
Using \eqref{EQ:21} we obtain $\frac{\kappa_{j}}{a_{j}}>1/\overline{\alpha}$
for all $j\in\{1,\dots,d\}$ and hence $\frac{a_{j}}{\kappa_{j}}\frac{1}{a_{i}}<\frac{\overline{\alpha}}{a_{i}}=\alpha_{i}$
for all $i,j\in\{1,\dots,d\}$. Hence we find $\eta\in(0,1)$ and
$c_1,\dots, c_d > 0$ such that, for all $i,j\in\{1,\dots,d\}$,
\[
0<\left(1+\frac{1}{\tau}\right)\frac{\eta}{a_{i}}<1\wedge\gamma_{*,i},\qquad\frac{a_{j}}{\kappa_{j}}\frac{1}{a_{i}}< c_i < \alpha_{i}\left(1-\frac{\eta}{a_{i}}\right).
\]
Define
\[
\lambda=\min_{i,j\in\{1,\dots,d\}}\left\{ c_i\delta a_{i},\ a_{i}-\eta-\frac{a_{i}c_i}{\alpha_{i}},\ \eta\left(c_i a_{i}\frac{\kappa_{j}}{a_{j}}-1\right)\right\} >0.
\]
Let $\phi\in C_{b}^{\eta,a}(\R^{d})$. By Proposition \ref{PROP:01}
we obtain, for $h\in[-1,1]$, $\e=|h|^{c_i}(1\wedge t)$ and $i\in\{1,\dots,d\}$,
\begin{align*}
\left|\E\left[\rho(X(t))\Delta_{he_{i}}\phi(X(t))\right]\right| & \leq C\|\phi\|_{C_{b}^{\eta,a}}\left(|h|^{\eta/a_{i}}\e^{\delta}+|h|\e^{-1/\alpha_{i}}+\max\limits _{j\in\{1,\dots,d\}}\e^{\eta\kappa_{j}/a_{j}}\right)\\
 & \leq\frac{C\|\phi\|_{C_{b}^{\eta,a}}}{(1\wedge t)^{1/\alpha_{i}}}\left(|h|^{\eta/a_{i}+c_i\delta}+|h|^{1-c_i /\alpha_{i}}+\max\limits _{j\in\{1,\dots,d\}}|h|^{c_i\eta\kappa_{j}/a_{j}}\right)\\
 & =\frac{C\|\phi\|_{C_{b}^{\eta,a}}}{(1\wedge t)^{1/\alpha_{i}}}|h|^{\eta/a_{i}}\left(|h|^{c_i \delta}+|h|^{1-\eta/a_{i}-c_i /\alpha_{i}}+\max\limits _{j\in\{1,\dots,d\}}|h|^{c_i \eta\kappa_{j}/a_{j}-\eta/a_{i}}\right)\\
 & \leq\frac{C\|\phi\|_{C_{b}^{\eta,a}}}{(1\wedge t)^{1/\alpha_{i}}}|h|^{(\eta+\lambda)/a_{i}}.
\end{align*}
The assertion now follows from Lemma \ref{LEMMA:02}. \end{proof}
By inspection of the proof, we obtain the following extension. 
\begin{Remark}\label{REMARK:01}
Estimate \eqref{EQ:22} can be replaced by the integrability condition
\[
\sup\limits _{t\in[0,T]}\E[\rho(X(t))^{-1}]<\infty,\ \ \forall T>0.
\]
In such a case $X(t)$ has, for $t>0$, a density on $\R^{d}$ (not
only on $\Gamma$). 
\end{Remark}

\section{Application to multi-type CBI processes}

\subsection{Proof of Theorem \ref{MAINTHEOREM}}

Our aim is to show that Theorem \ref{TH:00} is applicable. Let us
first show that \eqref{SDE:CBI} is a particular case of \eqref{EQ:03}.
Indeed, letting $b(x):=\beta+\widetilde{B}x$ and $\sigma(x)=\mathrm{diag}(\sqrt{2c_{1}x_{1}},\dots,\sqrt{2c_{d}x_{d}})$
we see that the first two terms have the desired form. Concerning
the jumps let $E=\R_{+}^{d}\times\R_{+}\times\{1,\dots,d+1\}$ and
set $E_{0}=\{z\in\R_{+}^{d}\ |\ |z|\leq1\}\times\R_{+}\times\{1,\dots,d\}$,
$E_{2}=\{z\in\R_{+}^{d}\ |\ |z|>1\}\times\R_{+}\times\{1,\dots,d\}$,
$E_{1}=\R_{+}^{d}\times\R_{+}\times\{d+1\}$. Define the corresponding intensity measure $m(d\xi)$, where $\xi=(z,r,k)\in E$, by
\begin{align*}
m(d\xi)=\sum\limits _{j=1}^{d}\mu_{j}(dz)dr\delta_{j}(dk)+\nu(dz)\delta_{0}(dr)\delta_{d+1}(dk).
\end{align*}
Finally choose $\sigma_{i}^{1}(x,\xi)=z_{i}$ and
\begin{align*}
\sigma_{i}^{0}(x,\xi)=z_{i}\1_{\{r\leq x_{k}\}}\1_{\R_{+}}(x_{k})\1_{\{1,\dots,d\}}(k),\qquad\sigma_{i}^{2}(x,\xi)=z_{i}\1_{\{r\leq x_{k}\}}\1_{\R_{+}}(x_{k})\1_{\{1,\dots,d\}}(k).
\end{align*}
Then it is not difficult to see that \eqref{SDE:CBI} is equivalent
in law to \eqref{EQ:03} with paramters defined above.
It is easily seen from the It\^{o} formula that both equations pose the same martingale problem.
Hence they describe the same law, which is sufficient for our purposes.
Let us show that conditions (B1) – (B4) are satisfied. Indeed (B1) is satisfied
for $J_{i}(b)=\{1,\dots,d\}$ and $\theta_{i}(b)=1$. Concerning condition
(B2) we see that
\begin{align*}
\int\limits _{E_{0}}|\sigma_{i}^{0}(x,\xi)-\sigma_{i}^{0}(y,\xi)|^{2}m(d\xi) 
 &=\sum\limits _{k=1}^{d}\int\limits _{|z|\leq1}\int\limits _{0}^{\infty}|\1_{\{r\leq x_{k}\}}\1_{\R_{+}}(x_{k})-\1_{\{r\leq y_{k}\}}\1_{\R_{+}}(y_{k})|z_{i}^{2}m(dz,dr,\{k\})\\
 & =\sum\limits _{k=1}^{d}\int\limits _{|z|\leq1}z_{i}^{2}\mu_{k}(dz)|x_{k}-y_{k}|\\
 & \leq\max\limits _{j\in\{1,\dots,d\}}\int\limits _{|z|\leq1}|z|^{2}\mu_{j}(dz)\sum\limits _{k=1}^{d}|x_{k}-y_{k}|
\end{align*}
and hence we may choose $J_{i}(\sigma^{0})=\{1,\dots,d\}$, $\theta_{i}(\sigma^{0})=\frac{1}{2}$
and $\gamma_{i}(\sigma^{0})=2$. For the integral against $\sigma^{1}$
we obtain $\int_{E_{1}}|\sigma_{i}^{1}(x,\xi)-\sigma_{i}^{1}(y,\xi)|m(d\xi)=0$,
i.e. $J_{i}(\sigma^{1})=\emptyset$, $\theta_{i}(\sigma^{1})=1$ and
$\gamma_{i}(\sigma^{1})=1$. In the same way we show that
\begin{align*}
\int\limits _{E_{2}}|\sigma_{i}^{2}(x,\xi)-\sigma_{i}^{2}(y,\xi)|^{1+\tau}m(d\xi)\leq\max\limits _{j\in\{1,\dots,d\}}\int\limits _{|z|>1}|z|^{1+\tau}\mu_{j}(dz)\sum\limits _{k=1}^{d}|x_{k}-y_{k}|,
\end{align*}
i.e. $J_{i}(\sigma^{2})=\{1,\dots,d\}$, $\theta_{i}(\sigma^{2})=\frac{1}{1+\tau}$
and $\gamma_{i}(\sigma^{2})=1+\tau$. This shows that (B2) is satisfied.
Condition (B3) is clearly satisfied with $J_{i}(\sigma)=\{i\}$ and
$\theta_{i}(\sigma)=\frac{1}{2}$. For the noise part \eqref{EQ:07}
appearing in condition (B4) we obtain
\begin{align*}
L_{i}^{x}(t) & =\1_{\R_{+}}(x_{i})\sqrt{2c_{i}x_{i}}B_{i}(t)+\int\limits _{0}^{t}\int\limits _{E_{0}}z_{i}\1_{\{r\leq x_{k}\}}\1_{\R_{+}}(x_{k})\widetilde{N}(du,d\xi)+\int\limits _{0}^{t}\int\limits _{E_{1}}z_{i}N(du,d\xi)\\
 & =\1_{\R_{+}}(x_{i})\sqrt{2c_{i}x_{i}}B_{i}(t)+\sum\limits _{j=1}^{d}\int\limits _{0}^{t}\int\limits _{|z|\leq1}\int\limits _{\R_{+}}z_{i}\1_{\{r\leq x_{j}\}}\1_{\R_{+}}(x_{j})\widetilde{N}_{j}(du,dz,dr)+\int\limits _{0}^{t}\int\limits _{\R_{+}^{d}}z_{i}N_{\nu}(du,dz),
\end{align*}
where $N_{\nu},N_{1},\dots,N_{d}$ are given as in \eqref{SDE:CBI}
and the second equality holds in law. Hence $L_{i}^{x}$ given by
\eqref{EQ:07} is precisely \eqref{EQ:01}. In particular, \eqref{MAIN:ASSUMPTION}
is precisely (B4) with $\rho(x)=\min\{x_{1},\dots,x_{d}\}\1_{\R_{+}^{d}}(x)$.
Observe that $G_{ji}(t)$ satisfies
\[
G_{ji}(t)\leq C\left(1+\sup\limits _{u\in[0,t]}\E[|X(u)|]+\sup\limits _{u\in[0,t]}\E[|X(u)|^{1+\tau}]\right)\leq C\left(1+\sup\limits _{u\in[0,t]}\E[|X(u)|^{1+\tau}]\right),
\]
i.e. it suffices to show that the right-hand side is finite. However,
in view of assumption (b) from Theorem \ref{MAINTHEOREM}, this property
can be classically shown by Gronwall. Note that, by $\gamma_{i}(\sigma^{1})=1$,
one has $\gamma_{*,i}=1$ and hence \eqref{ADMISSIBLE} is satisfied.
Next observe that $\gamma_{i}=2$ and hence \eqref{EQ:22} follows
from
\begin{align}
\E[|\rho(X(t))-\rho(X(t-\e))|]\leq\sum\limits _{j=1}^{d}\E[|X_{j}(t)-X_{j}(t-\e)|^{1/\alpha_{j}}]\leq C\sum\limits _{j=1}^{d}\e^{\frac{1}{2\alpha_{j}}}\leq C\e^{1/4},\label{EQ:08}
\end{align}
where we have used Lemma \ref{LEMMA:01} which is applicable due to
$\gamma_{*,j}=1\geq\frac{3}{4}>1/\alpha_{j}$. Finally, we have $\kappa_{i}=\frac{3}{4}$
and hence \eqref{EQ:21} is equivalent to $\alpha_{i}>\frac{4}{3}$,
which proves the assertion.

\subsection{Proof of Theorem \ref{MAINTHEOREM1}}

We proceed similarly to the previous case. Namely, observe that \eqref{SDE:CBI},
with $c_{1}=\dots=c_{d}=0$, is equivalent in law to \eqref{EQ:03}
for the particular choice $\sigma(x)=0$ and $b,E,E_{0},E_{1},E_{2},m,\sigma^{0},\sigma^{1},\sigma^{2}$
the same as in the proof of Theorem \ref{MAINTHEOREM}. Conditions
(B1) – (B3) are satisfied for $J_{i}(\sigma)=J_{i}(\sigma^{1})=\emptyset$,
$J_{i}(\sigma^{0})=J_{i}(\sigma^{2})=J_{i}(b)=\{1,\dots,d\}$, $\theta_{i}(b)=1$,
$\theta_{i}(\sigma^{0})=\frac{1}{\gamma_{0}}$, $\theta_{i}(\sigma^{1})=1$,
$\theta_{i}(\sigma^{2})=\frac{1}{1+\tau}$, $\theta_{i}(\sigma)=1$,
$\gamma_{i}(\sigma^{0})=\gamma_{0}$, $\gamma_{i}(\sigma^{1})=1$
and $\gamma_{i}(\sigma^{2})=1+\tau$. The noise part \eqref{EQ:07}
appearing in condition (B4) is precisely \eqref{EQ:01}, i.e. (B4)
follows from condition (A) with $\rho(x)=\rho_{I}(x)$. Estimating
$G_{ji}$ as before, we see that, for $I=\emptyset$ and hence $\rho_{\emptyset}=1$,
we may take $\tau=0$. Condition \eqref{ADMISSIBLE} can be shown
as in the proof of Theorem \ref{MAINTHEOREM}. For \eqref{EQ:22}
we obtain
\[
\E[|\rho_{I}(X(t))-\rho_{I}(X(t-\e))|]\leq\sum\limits _{j\in I}\E[|X_{j}(t)-X_{j}(t-\e)|^{1/\alpha_{j}}].
\]
Since, for $j\in I$, we have $\alpha_{j}\geq1=\gamma_{*,j}$, we
may proceed exactly as in \eqref{EQ:08}. Finally, we have $\gamma_{i}=\gamma_{i}^{*}=\gamma_{0}$
and hence $\kappa_{i}=\frac{1}{\gamma_{0}}\left(1+\frac{1}{\gamma_{0}}\right)$.
Thus \eqref{EQ:21} is equivalent to $\alpha_{i}>\frac{\gamma_{0}}{1+\gamma_{0}}\gamma_{0}$,
which proves the assertion.

\subsection{Proof of Theorem \ref{MAINTHEOREM2}}

We proceed similarly to the previous cases. Namely, \eqref{SDE:CBI}
is equivalent in law to \eqref{EQ:03} for the same choice as in the
proof of Theorem \ref{MAINTHEOREM1}. A simple computation shows that
conditions (B1) – (B3) are satisfied for $J_{i}(\sigma)=J_{i}(\sigma^{1})=\emptyset$,
$J_{i}(\sigma^{0})=J_{i}(\sigma^{2})=\{i\}$, $J_{i}(b)=\{1,\dots,d\}$,
$\theta_{i}(b)=1$, $\theta_{i}(\sigma^{0})=\frac{1}{\gamma_{0}^{i}}$,
$\theta_{i}(\sigma^{1})=1$, $\theta_{i}(\sigma^{2})=\frac{1}{1+\tau_{i}}$,
$\theta_{i}(\sigma)=1$, $\gamma_{i}(\sigma^{0})=\gamma_{0}^{i}$,
$\gamma_{i}(\sigma^{1})=1$ and $\gamma_{i}(\sigma^{2})=1+\tau_{i}$.
The noise part \eqref{EQ:07} appearing in condition (B4) is precisely
\eqref{EQ:01}, i.e. (B4) follows from condition (A) with $\rho(x)=\rho_{I}(x)$.
The function $G_{ji}$ can be estimated exactly as before (here we
need that $J_{i}(\sigma^{0})=J_{i}(\sigma^{2})=\{i\}$). Using $\gamma_{*,j}=1$
we see that \eqref{ADMISSIBLE} is satisfied. Condition \eqref{EQ:22}
can be shown in the same way as in the proof of Theorem \ref{MAINTHEOREM1}.
Finally, we have $\gamma_{i}=\gamma_{0}^{i}$, thus $\gamma_{i}^{*}=\max\{\gamma_{0}^{1},\dots,\gamma_{0}^{i}\}=:\gamma^{*}$,
$\kappa_{i}=\frac{1}{\gamma_{0}^{i}}\left(1+\frac{1}{\gamma^{*}}\right)$.
Hence \eqref{EQ:21} is equivalent to $\alpha_{i}>\frac{\gamma^{*}}{1+\gamma^{*}}\gamma_{0}^{i}$,
which proves the assertion.

\section{On the smoothing property (A)}
The following is due to \citep[Lemma 3.3]{DF13}. 
\begin{Proposition}\label{PROP:03}
Let $Z$ be a L\'evy process with L\'evy measure $m$ and symbol
\[
\Psi_{m}(\xi)=\int\limits _{\R^{d}}\left(1+i\xi\cdot z\1_{\{|z|\leq1\}}-e^{i\xi\cdot z}\right)m(dz).
\]
Suppose that there exist $\alpha\in(0,2]$ and $c,C>0$ with
\begin{align*}
c|\xi|^{\alpha}\leq\mathrm{Re}(\Psi_{m}(\xi))\leq C|\xi|^{\alpha},\ \ \forall\xi\in\R^{d},\ \ |\xi|\gg1.
\end{align*}
Then for each $t>0$, $Z(t)$ has a smooth density
$f_{t}$ and there exists a constant $C>0$ such
that
\[
\|\nabla f_{t}\|_{L^{1}(\R^{d})}\leq Ct^{-1/\alpha},\ \ t>0.
\]
\end{Proposition} 
Below we provide two sufficient conditions for (A). 
Our first result is a more general version of Example \ref{EXAMPLE:00}.(a).
\begin{Lemma}\label{LEMMA:SMOOTHING}
Define $I_{1}=\left\{ j\in\{1,\dots,d\}\ |\ c_{j}>0\right\} $ and
let $I_{2}:=\{1,\dots,d\}\backslash I_{1}$. Suppose that, for each
$j\in I_{2}$, there exists a L\'evy measure $\widetilde{\mu}_{j}$
on $\R_{+}$ with $\widetilde{\mu}_{j}(\{0\})=0$ and another Lévy
measure $\mu'$ on $\R_{+}^{d}$ with $\mu_{j}'(\{0\})=0$
satisfying \eqref{EQ:00} such that
\[
\mu_{j}(dz)=\widetilde{\mu}_{j}(dz_{j})\otimes\prod\limits _{k\neq j}\delta_{0}(dz_{k})+\mu'_{j}(dz).
\]
Moreover, assume that there exists $\alpha_{j}\in(0,2)$ and constants
$c,C>0$ with
\begin{align*}
c|\lambda|^{\alpha_{j}}\leq\int\limits _{|z|\leq1}\left(1-\cos(\lambda\cdot z)\right)\widetilde{\mu}_{j}(dz)\leq c|\lambda|^{\alpha_{j}},\ \ \lambda\in\R,\ \ |\lambda|\gg1.
\end{align*}
Then (A) is satisfied for $I=\{1,\dots,d\}$ and $\alpha_{j}=2\1_{I_{1}}(j)+\alpha_{j}\1_{I_{2}}(j)$.
\end{Lemma}
\begin{proof}
Fix $x\in\R_{+}^{d}$ such that $x_{1},\dots,x_{d}>0$. Write $L^{x}(t)=L_{1}^{x}(t)+L_{2}^{x}(t)$
where $L_{1}^{x},L_{2}^{x}$ are independent Lévy processes with symbols
\begin{align*}
\Psi_{x}^{1}(\lambda) & =\sum\limits _{j\in I_{1}}2c_{j}x_{j}\lambda_{j}^{2}+\sum\limits _{j\in I_{2}}x_{j}\int\limits _{(0,1)}\left(1+i\lambda_{j}\cdot z-e^{i\lambda_{j}\cdot z}\right)\widetilde{\mu}_{j}(dz),\\
\Psi_{x}^{2}(\lambda) & =\sum\limits _{j\in I_{2}}x_{j}\int\limits _{|z|\leq1}\left(1+i\lambda\cdot z-e^{i\lambda\cdot z}\right)\mu'_{j}(dz)\\
 & \ \ \ +\sum\limits _{j\in I_{1}}x_{j}\int\limits _{|z|\leq1}\left(1+i\lambda\cdot z-e^{-\lambda\cdot z}\right)\mu_{j}(dz)+\int\limits _{\R_{+}^{d}}\left(1-e^{i\lambda\cdot z}\right)\nu(dz).
\end{align*}
Then $g_{t}^{x}=f_{t}^{1}\ast f_{t}^{2}$, where $f_{t}^{j}$ is the
infinite divisible distribution of $L_{i}^{x}$, $i\in\{1,2\}$.
Observe that, for $\lambda\in\R^{d}$ sufficiently large and $\delta=\min\{\alpha_{j}\ |\ j\in I_{2}\}\wedge2$,
\[
\mathrm{Re}(\Psi_{x}^{1}(\lambda))\geq C\min\{x_{1},\dots,x_{d}\}|\lambda|^{\delta}.
\]
Hence $f_{t}^{1}$ has a smooth density, and thus also $g_{t}^{x}$
has a smooth density. Let $(B_{j})_{j\in I_{1}}$ be a collection
of independent one-dimensional Brownian motions and let $(Z_{j})_{j\in I_{2}}$
be a collection of independent one-dimensional Lévy processes with
symbols
\[
\Psi_{Z_{j}}(\lambda)=\int\limits _{(0,1)}\left(1+i\lambda\cdot z-e^{i\lambda\cdot z}\right)\widetilde{\mu}_{j}(dz),\ \ \lambda\in\R,\ \ j\in I_{2}.
\]
All these processes are supposed to be mutually independent. Then
$L_{1}^{x}$ satisfies in law
\begin{align*}
L_{1}^{x}(t) &= \sum\limits _{j\in I_{1}}e_{j}B_{j}(2c_{j}x_{j}t)+\sum\limits _{j\in I_{2}}e_{j}Z_{j}(x_{j}t)
\end{align*}
and hence $f_{t}^{1}(z)=\prod_{j\in I_{1}}h_{2c_{j}x_{j}t}(z_{j})\cdot\prod_{j\in I_{2}}\widetilde{h}_{x_{j}t}^{j}(z_{j})$,
where $h_{t}(z)$ is the gaussian density of $B_{j}(t)$ and $\widetilde{h}_{t}^{j}(z)$
is the smooth density of $Z_{j}(t)$. By Proposition \ref{PROP:03}
we obtain
\[
\int\limits _{\R}\left|\frac{\partial h_{t}(z)}{\partial z}\right|dz\leq Ct^{-1/2},\qquad\int\limits _{\R}\left|\frac{\partial\widetilde{h}_{t}^{j}(z)}{\partial z}\right|dz\leq Ct^{-1/\alpha_{j}},\ \ t>0.
\]
Thus we obtain, for $j\in I_{1}$,
\begin{align*}
\int\limits _{\R^{d}}\left|\frac{\partial f_{t}^{1}(z)}{\partial z_{j}}\right|dz\leq\frac{C}{\sqrt{x_{j}}}t^{-1/2}\leq\frac{C}{\rho(x)}t^{-1/2},\ \ t>0,
\end{align*}
and similarly, for $j\in I_{2}$,
\begin{align*}
\int\limits _{\R^{d}}\left|\frac{\partial f_{t}^{1}(z)}{\partial z_{j}}\right|dz\leq\frac{C}{x_{j}^{1/\alpha_{j}}}t^{-1/\alpha_{j}}\leq\frac{C}{\rho(x)}t^{-1/\alpha_{j}},\ \ t>0.
\end{align*}
The assertion follows from
\begin{align*}
\int\limits _{\R^{d}}\left|g_{t}^{x}(z+he_{j})-g_{t}^{x}(z)\right|dz\leq|h|\int\limits _{\R^{d}}\left|\frac{\partial g_{t}^{x}(z)}{\partial z_{j}}\right|dz\leq|h|\int\limits _{\R^{d}}\left|\frac{\partial f_{t}^{1}(z)}{\partial z_{j}}\right|dz,\ \ j\in\{1,\dots,d\}.
\end{align*}
\end{proof}
It is also possible to obtain the smoothing property (A) from the jump measure of the immigration mechanism. 
Our second result is a more general version of Example \ref{EXAMPLE:00}.(c).
\begin{Lemma}\label{LEMMA:SMOOTHING1} Suppose
that there exists $\alpha\in(0,1)$ and constants $c,C>0$ such that
\[
c|\lambda|^{\alpha}\leq\int\limits _{\R_{+}^{d}}\left(1-\cos(\lambda\cdot z)\right)\nu(dz)\leq C|\lambda|^{\alpha},\ \ |\lambda|\gg1.
\]
Then (A) is satisfied for $\alpha=\alpha_{1}=\dots=\alpha_{d}$ and
$I=\emptyset$. 
\end{Lemma}
\begin{proof}
Write $L^{x}=L_{1}^{x}+L_{2}^{x}$ where $L_{j}^{x}$ are Lévy processes
with symbols
\begin{align*}
\Psi_{x}^{1}(\lambda) & =\int\limits _{\R_{+}^{d}}\left(1-e^{i\lambda\cdot z}\right)\nu(dz),\\
\Psi_{x}^{2}(\lambda) & =\sum\limits _{j=1}^{d}2c_{j}x_{j}\1_{\R_{+}}(x_{j})\lambda_{j}^{2}+\sum\limits _{j=1}^{d}x_{j}\1_{\R_{+}}(x_{j})\int\limits _{|z|\leq1}\left(1+i\lambda\cdot z-e^{i\lambda\cdot z}\right)\mu_{j}(dz).
\end{align*}
Then $g_{t}^{x}=f_{t}^{1}\ast f_{t}^{2}$, where $f_{t}^{j}$ is the
distribution of $L_{j}^{x}$, $j\in\{1,2\}$. Using Proposition \ref{PROP:03}
we see that \textcolor{magenta}{{}
\[
{\normalcolor \int\limits _{\R^{d}}|\nabla g_{t}^{x}(z)|dz\leq\int\limits _{\R^{d}}|\nabla f_{t}^{1}(z)|dz\leq Ct^{-1/\alpha},\ \ t\to0.}
\]
}
This proves the assertion.
\end{proof}

\section{Some examples}
In this section we provide some simple examples showing how our main
results from Section 2 can be applied. Let $(c,\beta,B,\mu,\nu)$
be admissible parameters with $\nu=0$ and suppose that there exist
$\alpha_{1},\dots,\alpha_{d}\in(1,2)$ such that
\[
\mu_{k}(dz)=\frac{dz_{k}}{z_{k}^{1+\alpha_{k}}}\otimes\prod\limits _{j\neq k}\delta_{0}(dz_{j})+\mu_{k}'(dz),\ \ k\in\{1,\dots,d\},
\]
where $\mu_{k}'$ are Lévy measures on $\R_{+}^{d}$ satisfying $\mu_{k}'(\{0\})=0$
and \eqref{EQ:00}. Then we obtain the following:
\begin{enumerate}
\item[(a)] Theorem \ref{MAINTHEOREM} is applicable, provided $\alpha_{1},\dots,\alpha_{d}>\frac{4}{3}$
and $\mu_{k}'$ integrates $\1_{\{|z|>1\}}|z|^{1+\tau}$, for some
$\tau\in(0,1)$ and all $k\in\{1,\dots,d\}$.
\item[(b)] If $c_{1}=\dots=c_{d}=0$, then Theorem \ref{MAINTHEOREM1} is applicable,
provided
\begin{align}
\min\{\alpha_{1},\dots,\alpha_{d}\}>\frac{\max\{\alpha_{1},\dots,\alpha_{d}\}^{2}}{1+\max\{\alpha_{1},\dots,\alpha_{d}\}},\label{EQ:06}
\end{align}
and $\mu_{k}'$ integrates $\1_{\{|z|>1\}}|z|^{1+\tau}$, for some
$\tau\in(0,1)$ and all $k\in\{1,\dots,d\}$. Note that \eqref{EQ:06}
is weaker than $\max\{\alpha_{1},\dots,\alpha_{d}\}>\frac{4}{3}$.
\item[(c)] Suppose that $c_{1}=\dots=c_{d}=0$ and $\mu_{k}'=0$. Then Theorem
\ref{MAINTHEOREM2} is applicable. Note that the corresponding multi-type
CBI process can also be obtained as the pathwise
unique strong solution to the Lévy driven stochastic equation
\begin{align*}
X_{i}(t) & =X_{i}(0)+\int\limits _{0}^{t}\left(\beta_{i}+\sum\limits _{j=1}^{d}b_{ij}X_{j}(s)\right)ds+\int\limits _{0}^{t}X_{i}(s-)^{1/\alpha_{i}}dZ_{i}(s),
\end{align*}
where $Z_{1},\dots,Z_{d}$ are independent one-dimensional Lévy processes
with symbols
\[
\Psi_{k}(\xi)=\int\limits _{0}^{\infty}\left(1+i\xi z-e^{i\xi z}\right)\frac{dz}{z^{1+\alpha_{k}}},\ \ \xi\in\R,\ \ k\in\{1,\dots,d\}.
\]
\end{enumerate}
We remark that the above statements in (a) - (c)
also hold for $\nu\neq0$, provided $\int_{|z|>1}|z|^{1+\tau}\nu(dz)<\infty$,
for some $\tau\in(0,1)$. Below we provide one example, where existence of a density
is deduced from the smoothing property of the immigration mechanism. 
\begin{Example} Let $(c,\beta,B,\mu,\nu)$
be admissible parameters with $c_{1}=\dots=c_{d}=0$, $\mu_{1},\dots,\mu_{d}$
are such that, for some $\gamma_{0}\in\left(1,\frac{1+\sqrt{5}}{2}\right)$,
\[
\int\limits _{\R_{+}^{d}}\left( |z|^{\gamma_{0}}\1_{\{|z|\leq1\}} + |z| \1_{ \{ |z| > 1\}} \right)\mu_{k}(dz)<\infty,\ \ k\in\{1,\dots,d\},
\]
and the immigration mechanism is given by
\[
\nu(dz)=\1_{\left\{ z\in\R_{+}^{d}\ |\ |z|\leq 1\right\}} (z)\frac{dz}{|z|^{d+\alpha}}+\nu'(dz),\ \ \alpha\in(0,1),
\]
where $\nu'$ is any measure supported on $\R_{+}^{d}$ satisfying
$\nu'(\{0\})=0$ and $\int_{\R_{+}^{d}}|z|\nu'(dz)<\infty$. 
Then Theorem \ref{MAINTHEOREM1} is applicable
with $I=\emptyset$ and $\alpha_{1}=\dots=\alpha_{d}=\alpha$, provided
$\alpha>\frac{\gamma_{0}}{1+\gamma_{0}}\gamma_{0}$. 
\end{Example}

\section{Appendix}
Below we prove some simple estimates on the moments of stochastic
integrals with respect to Poisson random measures. Similar results
for the Lévy noise case were obtained in \citep[Lemma 5.2]{DF13}.
\begin{Lemma}\label{LEMMA:00} Let $N(du,dz)$ be a Poisson random
measure with compensator $\widehat{N}(du,dz)=dum(dz)$ on $\R_{+}\times E$,
where $m(dz)$ is a $\sigma$-finite measure on some Polish space
$E$. The following assertions hold.
\begin{enumerate}
\item[(a)] Let $0<\eta\leq\gamma$ and $1\leq\gamma\leq2$. Then there exists
a constant $C>0$ such that, for any predictable process $H(u,z)$
and $0\leq s\leq t\leq s+1$,
\[
\E\left[\left|\int\limits _{s}^{t}\int\limits _{E}H(u,z)\widetilde{N}(du,dz)\right|^{\eta}\right]\leq C(t-s)^{\eta/\gamma}\sup\limits _{u\in[s,t]}\E\left[\int\limits _{E}|H(u,z)|^{\gamma}m(dz)\right]^{\eta/\gamma},
\]
provided the stochastic integral is well-defined.
\item[(b)] Let $0<\eta\leq\gamma\leq2$. Then there exists a constant $C>0$
such that, for any predictable process $H(u,z)$ and $0\leq s\leq t\leq s+1$,
\begin{align*}
\E\left[\left|\int\limits _{s}^{t}\int\limits _{E}H(u,z)N(du,dz)\right|^{\eta}\right] & \leq C(t-s)^{\eta/\gamma}\sup\limits _{u\in[s,t]}\E\left[\int\limits _{E}|H(u,z)|^{\gamma}m(dz)\right]^{\eta/\gamma}\\
 & \ \ \ +C\1_{\gamma\in(1,2]}(t-s)^{\eta/\gamma}\sup\limits _{u\in[s,t]}\E\left[\left(\int\limits _{E}|H(u,z)|m(dz)\right)^{\gamma}\right]^{\eta/\gamma},
\end{align*}
provided the stochastic integral is well-defined.
\end{enumerate}
\end{Lemma} \begin{proof} (a) If $\eta\geq1$, then by the BDG-inequality,
sub-additivity of $x\longmapsto x^{\frac{\gamma}{2}}$ and Hölder
inequality we obtain
\begin{align*}
 & \ \E\left[\left|\int\limits _{s}^{t}\int\limits _{E}H(u,z)\widetilde{N}(du,dz)\right|^{\eta}\right]\leq C\E\left[\left|\int\limits _{s}^{t}\int\limits _{E}|H(u,z)|^{2}N(du,dz)\right|^{\eta/2}\right]\\
 & \leq C\E\left[\left|\int\limits _{s}^{t}\int\limits _{E}|H(u,z)|^{\gamma}N(du,dz)\right|^{\eta/\gamma}\right]\leq C\E\left[\int\limits _{s}^{t}\int\limits _{E}|H(u,z)|^{\gamma}dum(dz)\right]^{\eta/\gamma}\\
 & \leq C(t-s)^{\frac{\eta}{\gamma}}\sup\limits _{u\in[s,t]}\E\left[\int\limits _{E}|H(u,z)|^{\gamma}m(dz)\right]^{\eta/\gamma}.
\end{align*}
If $0<\eta\leq1\leq\gamma\leq2$, then the Hölder inequality and previous
estimates imply
\begin{align*}
\E\left[\left|\int\limits _{s}^{t}\int\limits _{E}H(u,z)\widetilde{N}(du,dz)\right|^{\eta}\right] & \leq\E\left[\left|\int\limits _{s}^{t}\int\limits _{E}H(u,z)\widetilde{N}(du,dz)\right|\right]^{\eta}\\
 & \leq C(t-s)^{\frac{\eta}{\gamma}}\sup\limits _{u\in[s,t]}\E\left[\int\limits _{E}|H(u,z)|^{\gamma}m(dz)\right]^{\eta/\gamma}.
\end{align*}
(b) If $\gamma\in(0,1]$, then by sub-additivity of $x\longmapsto x^{\gamma}$
and Hölder inequality we get
\begin{align*}
 & \ \E\left[\left|\int\limits _{s}^{t}\int\limits _{E}H(u,z)N(du,dz)\right|^{\eta}\right]\leq\E\left[\left|\int\limits _{s}^{t}\int\limits _{E}|H(u,z)|^{\gamma}N(du,dz)\right|^{\eta/\gamma}\right]\\
 & \leq\E\left[\int\limits _{s}^{t}\int\limits _{E}|H(u,z)|^{\gamma}dum(dz)\right]^{\eta/\gamma}\leq(t-s)^{\frac{\eta}{\gamma}}\sup\limits _{u\in[s,t]}\E\left[\int\limits _{E}|H(u,z)|^{\gamma}m(dz)\right]^{\eta/\gamma}.
\end{align*}
If $\gamma\in(1,2]$, then
\begin{align*}
 & \ \E\left[\left|\int\limits _{s}^{t}\int\limits _{E}H(u,z)N(du,dz)\right|^{\eta}\right]\leq C\E\left[\left|\int\limits _{s}^{t}\int\limits _{E}H(u,z)\widetilde{N}(du,dz)\right|^{\eta}\right]+C\E\left[\left|\int\limits _{s}^{t}\int\limits _{E}H(u,z)dum(dz)\right|^{\eta}\right].
\end{align*}
The stochastic integral can be estimated by part (a), and the second
term by
\begin{align*}
\E\left[\left|\int\limits _{s}^{t}\int\limits _{E}H(u,z)dum(dz)\right|^{\eta}\right] & \leq\E\left[\left|\int\limits _{s}^{t}\int\limits _{E}H(u,z)dum(dz)\right|^{\gamma}\right]^{\eta/\gamma}\\
 & \leq(t-s)^{\eta}\sup\limits _{u\in[s,t]}\E\left[\left(\int\limits _{E}|H(u,z)|m(dz)\right)^{\gamma}\right]^{\eta/\gamma},
\end{align*}
which proves the assertion since $t-s\leq1$ and $\gamma\geq1$. \end{proof}

\begin{footnotesize}

 \bibliographystyle{alpha}
\bibliography{Bibliography}

\end{footnotesize}
\end{document}